\documentclass[12pt,a4paper]{amsart}

\usepackage{amssymb,xspace}
\usepackage{amstext}
\theoremstyle{plain}
\usepackage{amsbsy,amssymb,amsfonts,latexsym}
\usepackage[utf8]{inputenc}
\usepackage{xcolor}
\usepackage{graphicx} 
\usepackage{url}
\usepackage{a4wide}

\usepackage{enumerate}
\usepackage{graphics}

\newcommand{\veps}{\varepsilon}

\newcommand{\supp}{\text{supp}}
\def\QQ{\mathbb Q}

\def\RR{\mathbb R}
\def\R{\mathbb R}

\def\NN{\mathbb N}
\def\ZZ{\mathbb Z}

\def\CC{\mathbb C}
\def\C{\mathbb C}

\def\card{\textrm{card}}

\newtheorem{theorem}{Theorem}[section]
\newtheorem{lemma}[theorem]{Lemma}
\newtheorem{corollary}[theorem]{Corollary}
\newtheorem{proposition}[theorem]{Proposition}

\theoremstyle{definition}

\newtheorem{example}[theorem]{Example}

\theoremstyle{remark}
\newtheorem{remark}[theorem]{Remark}

\newtheorem{question}[theorem]{Question}

\begin{document}


\date{\today}

\title{Baire theorem and hypercyclic algebras}

\author[F. Bayart]{Fr\'ed\'eric Bayart}
\email{Frederic.Bayart@uca.fr}
\author[F. Costa J\'unior]{Fernando Costa J\'unior}
\email{Fernando.Vieira\underline{ }Costa\underline{ }Junior@uca.fr,}
\author[D. Papathanasiou]{Dimitris Papathanasiou}
\email{dpapath@bgsu.edu}
\thanks{The first and the second author were partially supported by the grant ANR-17-CE40-0021 of the French National Research Agency ANR (project Front)}
\address{Laboratoire de Mathématiques Blaise Pascal UMR 6620 CNRS, Université Clermont Auvergne, Campus universitaire des Cézeaux, 3 place Vasarely, 63178 Aubière Cedex, France.}

\subjclass{47A16}

\keywords{hypercyclic operators, weighted shifts, convolution operators, hypercyclic algebras}

\begin{abstract}
The question of whether a hypercyclic operator $T$ acting on a Fréchet algebra $X$ admits or not an algebra of hypercyclic vectors (but 0) has been addressed in the recent literature. In this paper we give new criteria and characterizations in the context of convolution operators acting on $H(\CC)$ and backward shifts acting on a general Fréchet sequence algebra.
Analogous questions arise for stronger properties like frequent hypercyclicity. 
In this trend we give a sufficient condition for a weighted backward shift to admit an upper frequently hypercyclic algebra and we find a weighted backward shift acting on $c_0$ admitting a frequently hypercyclic algebra for the coordinatewise product. The closed hypercyclic algebra problem is also covered.
\end{abstract}

\maketitle



\section{Introduction}

Among the many problems in linear dynamics, understanding the structure of the set of hypercyclic vectors is a major one. Let us introduce the relevant definitions.
Let $(X,T)$ be a linear dynamical system, namely $X$ is a topological vector space and $T$ is a bounded linear operator on $X$.
A vector $x\in X$ with dense orbit under $T$ is called a hypercyclic vector, and we denote by $HC(T)$ the set of hypercyclic vectors for $T$:
$$
HC(T)=\{x\in X: \{x, Tx, T^2x,\dots \} \,\, \mbox{is dense in} \,\, X\}.
$$
This set $HC(T)$ possesses interesting properties. When $X$ is a Baire space, its nonemptyness implies its residuality, preventing it from being a non-trivial proper linear subspace of $X$.
However, it is well known that, whenever $HC(T)$ is nonempty, then $HC(T)\cup \{0\}$ contains a
dense linear manifold (see \cite{bourdon}). In many cases (not always) $HC(T)\cup \{0\}$ contains  a closed, infinite dimensional linear subspace
(see \cite{GoRo, peterson,shkarin, menet}). These properties reflect that we are working in a linear space.

Suppose now that $X$ has a richer structure: it is an $F$-algebra, namely a metrizable and complete topological algebra. It is natural to ask whether $HC(T)\cup\{0\}$ also contains 
a non-trivial subalgebra of $X$. Such an algebra will be called a \emph{hypercyclic algebra}. The pioneering work in that direction has been done independently by Shkarin in \cite{shkarin} and 
by Bayart and Matheron in \cite{BM09}: they showed that the derivation operator $D:f\mapsto f'$, acting on the Fr\'echet algebra $H(\CC)$ of entire functions endowed with
the pointwise multiplication, supports a hypercyclic algebra. However, this is not the case for all hypercyclic operators acting on an $F$-algebra: 
for instance, as pointed out in \cite{ACPS07}, the translation operators, acting on $H(\CC)$, do not support a hypercyclic algebra. 
Recent papers (see e.g. \cite{Bayhcalg,bes1,BCP18,BP18,FalGre18,FalGre19}) give other examples of operators admitting a hypercyclic algebra. 

Our aim, in this paper, is to shed new light on this problem and to study how it interacts with popular problems arising in linear dynamics. We are particularly interested in two questions.
  
\subsection{Existence of hypercyclic algebras}

All examples in the literature of operators supporting a hypercyclic algebra are generalizations of $D$. There
are several ways to extend it. You may see $D$ as a special convolution operator acting on $H(\CC)$. By \cite{GoSh91}, such an operator may be written $\phi(D)$, where $\phi$ is an entire function with exponential type; if $\phi$ is not constant, then $\phi(D)$ is hypercyclic. When $|\phi(0)|<1$, the existence
of hypercyclic algebras is well-understood since \cite{Bayhcalg}: such an algebra does exist if and only if $\phi$ is not a multiple of an exponential function. When $|\phi(0)|=1$, sufficient conditions are given in \cite{Bayhcalg} or in \cite{BesErnstPrieto}
but almost nothing, except a very specific example, is known when $|\phi(0)|>1$. We partly fill this gap by proving 
the existence of a hypercyclic algebra when $\phi$ goes to zero along some half-line.

\begin{theorem}\label{thm:convintro}
Let $\phi$ be a nonconstant entire function with exponential type, not a multiple of an exponential function. 
Assume that $|\phi(0)|>1$ and that there exists some $w\in\CC$ such that $|\phi(tw)|\to 0$ as $t\to+\infty$. Then $\phi(D)$ supports a hypercyclic algebra.
\end{theorem}
In particular, we shall see that if $\phi(z)=P(z)e^z$ for some non-constant polynomial $P$, then $\phi(D)$ supports a hypercyclic algebra.

\smallskip

Another way to generalize $D$ is to see it as a weighted backward shift acting on $H(\CC)$ considered as a sequence space.
This was explored in \cite{FalGre18}. The general context is that of a Fr\'echet sequence algebra $X$. Precisely we assume that $X$ is a subspace of the space $\omega=\CC^{\NN_0}$ of all complex sequences, whose topology is induced by a 
non-decreasing sequence of seminorms $(\|\cdot\|_q)_{q\geq 1}$ and that $X$ is endowed with a product $\cdot$ such that, for all $x,y\in X$, all $q\geq 1$, 
$$\|x\cdot y\|_q\leq \|x\|_q\times \|y\|_q.$$
There are two natural products on a Fr\'echet sequence space: the coordinatewise product and the convolution or Cauchy product.
It is clear that $\ell_p$ and $c_0$ are Fréchet sequence algebras for the coordinatewise product, and that $\ell_1$ is also a Fréchet sequence algebra for the convolution product.
Endowing $H(\mathbb C)$ with 
\[ \left\| \sum_{n\geq 0}a_n z^n \right\|_q =\sum_{n\geq 0}|a_n| q^n \]
and $\omega$ with
\[ \left\| (x_n) \right\|_q=\sum_{n=0}^q |x_n|,\]
we also obtain that $H(\CC)$ and $\omega$ are Fréchet sequence algebras for both products (on $H(\mathbb C)$, the Cauchy product of $f$ and $g$ is nothing else but the product
of the two functions $f$ and $g$). Another interesting source of examples for us will be the sequence spaces $X=\{(x_n)\in\omega: \gamma_n x_n\to 0\}$ endowed with $\|x\|=\sup_n \gamma_n |x_n|$, where $(\gamma_n)\in\RR_+^{\NN_0}$. Provided $\gamma_n\geq 1$ for all $n$, $X$ is a Fr\'echet sequence algebra for the coordinatewise product. 

Given a sequence of nonzero complex numbers $w=(w_n)_{n\in\NN}$, the (unilateral) weighted backward shift $B_w$ with weight $w$ is defined by 
$$B_w(x_0,x_1,\dots)=(w_1x_1,w_2x_2,\dots).$$
The weight $w$ will be called \emph{admissible} (for $X$) if $B_w$ is a bounded operator on $X$. 
It is known that, provided the canonical basis $(e_n)$ is a Schauder basis of $X$, $B_w$ is hypercyclic if and only if there exists a sequence $(n_k)$ such that
for all $l\in\NN$, $\big((w_{l+1}\cdots w_{n_k+l})^{-1}e_{n_k+l}\big)$ goes to zero. 

Let us first assume that $X$ is a Fr\'echet algebra under the coordinatewise product. Under a supplementary technical condition on $X$, a sufficient condition on $w$ is given in \cite{FalGre18}
so that $B_w$ supports a hypercyclic algebra. It turns out that we shall give a very natural characterization
of this property when $X$ admits a continous norm. We recall that a Fr\'echet space $(X,(\|\cdot\|_q))$ admits a continuous norm if there exists a norm $\|\cdot\|:X\to\mathbb R$ that is continuous for the topology of $X$, namely there exists $q\in\NN$ and $C>0$ with $\|x\|\leq C\|x\|_q$ for all $x\in X$. In particular, for any $q$ large enough, $\|\cdot\|_q$ itself is a norm.

\begin{theorem}\label{thm:mainws}
Let $X$ be a Fr\'echet sequence algebra for the coordinatewise product and with a continuous norm. Assume that $(e_n)$ is a Schauder basis for $X$. Let also $B_w$ be a bounded weighted shift on $X$. The following assumptions are equivalent.
\begin{enumerate}[(i)]
\item $B_w$ supports a dense and not finitely generated hypercyclic algebra.
\item There exists a sequence of integers $(n_k)$ such that for all $\gamma>0$, for all $l\in\mathbb N$, $\big((w_{l+1}\cdots w_{n_k+l})^{-\gamma}e_{n_k+l}\big)$ tends to zero.
\end{enumerate}
\end{theorem}

In particular, this theorem implies that on $\ell_p$ or $c_0$, any hypercyclic weighted shift supports a hypercyclic algebra.

\smallskip

When $X$ is a Fr\'echet algebra for the Cauchy product, it is shown in \cite{FalGre18} that, under additional technical assumptions on $X$, $B_w$ supports a hypercyclic algebra as soon as it is mixing,
namely as soon as $(w_1\cdots w_n)^{-1}e_n$ tends to zero. We shall improve that theorem by showing that any hypercyclic backward shift on a Fréchet sequence algebra for the Cauchy product
supports a hypercyclic algebra. We will only require a supplementary assumption on $X$ (to be regular) which is less strong than the assumption required in \cite{FalGre18}. 

\begin{theorem}\label{thm:wscauchy}
 Let $X$ be a regular Fréchet sequence algebra for the Cauchy product and let $B_w$ be a bounded weighted shift on $X$. The following assertions are equivalent.
 \begin{enumerate}[(i)]
  \item $B_w$ is hypercyclic.
  \item $B_w$ supports a dense and not finitely generated hypercyclic algebra.
 \end{enumerate}
\end{theorem}

In particular, if we compare this statement with Theorem \ref{thm:mainws}, we see that, for the convolution product, we do not need extra assumptions on the weight, which was not the case
for the coordinatewise product. Even if the proofs of Theorems \ref{thm:mainws}, \ref{thm:wscauchy} share some similarities (the latter one being much more difficult), they also have strong
differences, the main one being that, under the coordinatewise product, any power of $x\in\omega$ keeps the same support, which is far from being the case if we work with the Cauchy product.

We also point out that this detailed study of the existence of hypercyclic algebras for weighted shifts has interesting applications. 
For instance, working with a bilateral shift, it will allow us to exhibit an invertible operator on a Banach algebra supporting a hypercyclic algebra and such that its inverse does not (see Example \ref{ex:inverse}).

\subsection{Frequently and upper frequently hypercyclic algebras}

Another fruitful subject in linear dynamics is frequent and upper frequent hypercyclicity. We say that $T$ is frequently hypercyclic (resp. upper frequently hypercyclic) if there exists a vector $x\in X$ such that, for all $U\subset X$ open and non-empty, the set $\{n\in\mathbb N: T^n x\in U\}$ has positive lower density (resp. positive upper density). Again, linearity allows to give a nice criterion to prove that an operator is (upper) frequently hypercyclic and gives rise to nice examples. 
For instance, if $X$ is a Fr\'echet sequence space and $B_w$ is a bounded weighted shift acting on $X$, it is known that the unconditional convergence of $\sum_{n\geq 1}(w_1\cdots w_n)^{-1}e_n$
implies that $B_w$ is frequently hypercyclic. Moreover, in some spaces (for instance, on $\ell_p$-spaces), this condition is even necessary for the upper frequent hypercyclicity of $B_w$.

Of course, it is natural to ask if a (upper) frequently hypercyclic operator defined on an $F$-algebra $X$ admits a (upper) frequently hypercyclic algebra, namely an algebra consisting only, except 0, of (upper) 
frequently hypercyclic vectors.
Falc\'o and Grosse-Erdmann have shown recently (\cite{FalGre19}) that this is not always the 
case: for instance, $\lambda B$, $\lambda>1$, acting on any $\ell_p$ space ($1\leq p<+\infty$) or on $c_0$, endowed with the coordinatewise product, does not admit a frequently hypercyclic algebra. 
Nevertheless, this leaves open the possibility for $\lambda B$, $\lambda>1$, to admit an upper frequently hypercyclic algebra. 

We shall give two general results implying that a weighted shift on a Fréchet sequence algebra admits an upper frequently hypercyclic algebra. 
The first one deals with Fréchet sequence algebras endowed with the coordinatewise product.
In view of Theorem \ref{thm:mainws}, the natural extension of the above result for the existence of an upper frequently hypercyclic vector
is to ask now for the unconditional convergence of the series $\sum_{n\geq 1}(w_1\cdots w_n)^{-1/m}e_n$ for all $m\geq 1$. This is sufficient! 

\begin{theorem}\label{thm:ufhcws}
Let $X$ be a Fr\'echet sequence algebra for the coordinatewise product and with a continuous norm. 
Assume that $(e_n)$ spans a dense subspace of $X$. Let also $B_w$ be a bounded weighted shift on $X$ such that, for all $m\geq 1$, 
$\sum_{n\geq 1}(w_1\cdots w_n)^{-1/m}e_n$ converges unconditionally. Then $B_w$ admits an upper frequently hypercyclic algebra.
\end{theorem}

In particular, for $\lambda>1$, $\lambda B$ admits on any $\ell_p$-space ($1\leq p<+\infty$)  and on $c_0$ an upper frequently hypercyclic algebra.
This last result was independently obtained by Falc\'o and Grosse-Erdmann in \cite{FalGre19} in a different context (they concentrate themselves on $\lambda B$ 
but allow different notions of hypercyclicity) and with a completely different proof.

Regarding Fréchet sequence algebras endowed with the convolution product, we also have been able to get a general statement (see the forthcoming Theorem \ref{thm:ufhcconvolution}). 
Its main feature is that we will only need the unconditional convergence of $\sum_{n\geq 1}(w_1\cdots w_n)^{-1}e_n$ and a technical condition to ensure the existence of an upper frequently hypercyclic algebra. As a corollary, we can state the following.
\begin{corollary}
\begin{enumerate}[(i)]
\item Let $X=\ell_1$ endowed with the convolution product and let $\lambda>1$. Then $\lambda B$
admits an upper frequently hypercyclic algebra.
\item Let $X=H(\CC)$ endowed with the convolution product. Then $D$ 
admits an upper frequently hypercyclic algebra.
\end{enumerate}
\end{corollary}


Coming back to our initial problem, we show that it is possible to exhibit a weighted shift supporting a frequently hypercyclic algebra.
The place to do this will be $c_0$ endowed with the coordinatewise product; of course, the weight sequence will be much more complicated than that of the Rolewicz operator.

\begin{theorem}\label{thm:afhcc0}
 There exists a weight $(w_n)$ such that $B_w$, acting on $c_0$ endowed with the coordinatewise product, supports a frequently hypercyclic algebra.
\end{theorem}

The proof of this theorem will need the construction of disjoint subsets of $\NN$ with positive lower density and with some other extra properties, which 
seems interesting by itself.

\subsection{Organization of the paper}
Up to now, there were two ways to produce hypercyclic algebras: by a direct construction (this is the method devised in \cite{shkarin} and in \cite{FalGre18})
or by using a Baire argument (this method was initiated in \cite{BM09}). In this paper, we improve the latter.
We first give in Section \ref{sec:criterion} a general result for the existence of a hypercyclic algebra, enhancing the main lemma proved in \cite{BM09}. 
This general theorem will be suitable to our new examples of operators supporting a hypercyclic algebra. 
Next, we adapt the Baire argument to produce upper frequently hypercyclic algebras as well. Since the set of frequently hypercyclic vectors is always meagre, Theorem \ref{thm:afhcc0} cannot be proved using such an argument; it follows from a careful construction both of the weight and of the algebra.

We finish in the last section by making some remarks and asking some questions. In particular, we give a negative answer to a question raised by Shkarin about the existence of a closed hypercyclic algebra for the derivation operator.

\subsection{Notations}
The symbol $\mathbb N$ will stand for the set of positive integers, whereas $\NN_0=\NN\cup\{0\}$.
We shall denote by $\mathcal P_f(A)$ the set of finite subsets of a given set $A$. 

For $x=\sum_{n=0}^{+\infty}x_n e_n\in \omega$, the support of $x$ is equal to $\supp(x)=\{n\in\NN_0: x_n\neq 0\}$. The notation $c_{00}$ will denote the set of
sequences in $\omega$ with finite support.

For $u\in X^d$ and $\alpha\in\NN_0^d$, 
$u^\alpha$ will mean $u_1^{\alpha_1}\cdots u_d^{\alpha_d}$. If $z$ is any complex number and $m\in\NN$, $z^{1/m}$ will denote any $m$th root of $z$.

When working on a Fréchet space $(X,\|\cdot\|_p)$, it is often convenient to endow $X$ with an $F$-norm $\|\cdot\|$ defining the topology of $X$ (see \cite[Section 2.1]{GePeBook}).
Such an $F$-norm can be defined by the formula
\[ \|x\|=\sum_{p=1}^{+\infty}\frac1{2^p}\min(1,\|x\|_p). \]
In particular, an $F$-norm satisfies the triangle inequality and the inequality
\begin{equation}\label{eq:fnorm}
\forall \lambda\in\CC,\ \forall x\in X,\ \|\lambda x\|\leq (|\lambda|+1)\|x\|,
\end{equation}
a property which replaces the positive homogeneity of the norm.

We finally recall some results on unconditional convergence in Fr\'echet spaces (see for instance \cite[Appendix A]{GePeBook}). A series $\sum_{n=0}^{+\infty}x_n$ 
in a Fréchet space $X$ is called unconditionally convergent if for any bijection $\pi:\NN_0\to\NN_0$, the series $\sum_{n=0}^{+\infty}x_{\pi(n)}$ is convergent.
This amounts to saying that, for any $\veps>0$, there is some $N\in\NN$ such that, whenever $\sup_n |\alpha_n|\leq 1$, the series $\sum_{n=0}^{+\infty}\alpha_n x_n$
converges and 
\[\left \|\sum_{n=N}^{+\infty}\alpha_n x_n\right\|<\veps.
\]


\section{A general criterion}
\subsection{A transitivity criterion to get hypercyclic algebras}
\label{sec:criterion}

We first give a general statement which may be thought of as a Birkhoff transitivity theorem for hypercyclic algebras.
This criterion will be the main ingredient for the results from Section \ref{sec:convolution} and some from Section \ref{sec:ws}.
A different version of this approach will give rise to a new criterion for the existence of upper frequently hypercyclic algebras in Section \ref{sec:fhc}.

\begin{theorem}\label{thm:generalcriterion}
 Let $T$ be a continuous operator on a separable commutative $F$-algebra $X$ and let $d\geq 1$. Assume that for any $A\subset \NN_0^d \backslash\{(0,\dots,0)\}$ finite and non-empty, 
   for any non-empty open subsets $U_1,\dots,U_d,V$ of $X$, for any neighbourhood $W$ of $0$, 
 there exist $u=(u_1,\dots,u_d)\in U_1\times\cdots\times U_d$, $\beta\in A$  and $N\geq 1$ such that $T^N(u^\beta)\in V$ and $T^N(u^\alpha)\in W $ for all $\alpha\in A$, $\alpha\neq\beta$.
 Then the set of $d$-tuples that generate a hypercyclic algebra for $T$ is residual in $X^d$. Moreover, if the assumptions are satisfied for all $d\geq 1$, then $T$ admits a dense and  not finitely generated hypercyclic algebra.
\end{theorem}

\begin{proof}
Let $(V_k)$ be a basis of open neighbourhoods of $X$. For $A\in \mathcal P_f(\NN_0^d)$, $A\neq\varnothing$, $(0,\dots,0)\notin A$, for $s,k\geq 1$, for $\beta\in A$,
define
\begin{align*}
 E(A, \beta,s)&=\left\{\sum_{\alpha\in A}\hat P(\alpha)z^\alpha\in\mathbb C[z_1,\dots,z_d]:\ \hat P(\beta)=1\textrm{ and }\sup_{\alpha\in A} |\hat P(\alpha)|\leq s\right\}\\
 \mathcal A(A,\beta,s,k)&=\left\{u\in X^d:\ \forall P\in E(A,\beta,s),\ \exists N\geq 1,\ T^N(P(u))\in V_k\right\}.
\end{align*}
The sets $E(A,\beta,s)$ are compact subsets of $\mathbb C[z_1,\dots,z_d]$.
By continuity of the maps $(u,P)\mapsto T^N(P(u))$, this implies that each set $\mathcal A(A,\beta,s,k)$ is open.
Let us show that, for all $A$, $k$ and $s$, $\bigcup_{\beta\in A}\mathcal A(A,\beta,s,k)$ is dense in $X^d$. Indeed,
pick $U_1,\dots,U_d$ non-empty open subsets of $X$. Let $V\subset V_k$ and $W=B(0,\veps)$ be a neighbourhood of $0$ such that
$V+B\big(0,(s+1)\card(A)\veps\big)\subset V_k$. The assumptions of the proposition give the existence of $u=(u_1,\dots,u_d)\in U_1\times\cdots\times U_d$, $\beta\in A$
and $N\geq 1$. We claim that $u$ belongs to $\mathcal A(A,\beta,s,k)$. Indeed, 
$$T^N P(u)=\sum_{\alpha\neq\beta} \hat P(\alpha)T^N(u^\alpha)+T^N(u^\beta)\in V+B\big(0,(s+1)\card(A)\veps\big)\subset V_k$$
(observe that we have used \eqref{eq:fnorm}). 
Hence, $\bigcap_{A,s,k}\bigcup_\beta \mathcal A(A,\beta,s,k)$ is a residual subset of $X^d$. Pick $u\in \bigcap_{A,s,k}\bigcup_\beta \mathcal A(A,\beta,s,k)$.



We show that for all non-zero polynomials $P\in \mathbb C[z_1,\dots,z_d]$ with $P(0)=0$, $P(u)$ belongs to $HC(T)$.
We set $A=\left\{\alpha:\ \hat P(\alpha)\neq 0\right\}$ and we first prove that $\bigcup_{\beta\in A}\textrm{Orb}\left(T,\frac1{\hat P(\beta)}P(u)\right)$ is dense. Let us fix some $k$
and let us set $s=\sup_{\alpha,\beta\in A} |\hat P(\alpha)|/|\hat P(\beta)|$. Let $\beta\in A$ be such that $u\in\mathcal A(A,\beta,s,k)$. Define $Q=P/\hat P(\beta)$.
Then $Q$ belongs to $E(A,\beta,s)$ so that there exists $N\geq 1$ satisfying
\[ T^N\left(\frac1{\hat P(\beta)}P(u)\right)=T^N \big(Q(u)\big)\in V_k. \]
By the Bourdon-Feldman theorem, we deduce that there is some $\beta_0\in A$ such that $\textrm{Orb}\left(T,\frac 1{\hat P(\beta_0)}P(u)\right)$ is dense in $X$. Since any non-zero
multiple of a hypercyclic vector remains hypercyclic, we finally deduce that $P(u)$ is a hypercyclic vector for $T$.


The modification to obtain dense and infinitely generated algebras is easy. For $A\in\mathcal P_f(\NN_0^d)$, $A\neq\varnothing$, $(0,\dots,0)\notin A$, we now let 
$$ \mathcal A(A,\beta,s,k)=\left\{u\in X^\mathbb N:\ \forall P\in E(A,\beta,s),\ \exists N\geq 1,\ T^N(P(u))\in V_k\right\}$$
and we still consider the set $\bigcap_{A,s,k}\bigcup_{\beta\in A}\mathcal A(A,\beta,s,k)$ where now the intersection runs over all non-empty and finite sets 
$A\subset \mathbb N^d\backslash\{(0,\dots,0)\}$ with $d\geq 1$ arbitrary. This intersection is still residual in $X^\mathbb N$. 
We also know from \cite{BP18} that the set of $u$ in $X^\NN$ that induce a dense algebra in $X$ is residual in $X^\NN$. Hence we may pick $u\in X^\NN$ belonging
to $\bigcap_{A,s,k}\bigcup_\beta\mathcal A(A,\beta,s,k)$ and inducing a dense algebra in $X$.
It is plain that for any non-zero polynomial $P$ with $P(0)=0$, $P(u)$ is hypercyclic for $T$.

It remains to show that the algebra generated by $u$ is not finitely generated. 
Assume on the contrary that it is generated by a finite number of $P_1(u),\dots,P_p(u)$. In particular, it is generated by a finite number
of $u_1,\dots,u_q$. Then there exists a polynomial $Q\in \mathbb C[z_1,\dots,z_q]$ such that $Q(0)=0$ and $u_{q+1}=Q(u_1,\dots,u_q)$.
Define $P(z)=z_{q+1}-Q(z)$. Then 
$P$ is a non-zero polynomial with $P(0)=0$. Nevertheless, $P(u)=0$, which contradicts the fact that $P(u)$ is a hypercyclic vector for $T$.
\end{proof}

\begin{remark}
 Theorem \ref{thm:generalcriterion} remains true if the algebra is not commutative. This is clear if $d=1$. For the remaining cases,
 we have to replace in the proof polynomials in $d$ commutative variables by polynomials in $d$ non-commutative variables.
 Details are left to the reader.
\end{remark}

We point out that, unlike \cite[Lemma 3.1]{BP18}, in the previous theorem, the index $\beta$ may depend on $A$, $U_1,\dots,U_d,V$ and $W$. We will never use this possibility:
we will only need that $\beta$ may depend on $A$ and we will denote $\beta=\beta_A$. For this particular case, we could give an easier proof
avoiding the use of the Bourdon-Feldman theorem (see the proof of \cite[Lemma 3.1]{BP18}).

Let us give a couple of corollaries. The first one comes from \cite[Remark 5.28]{BM09} and was the key lemma in \cite{BM09}, \cite{BCP18} or \cite{Bayhcalg} to get hypercyclic algebras.

\begin{corollary}\label{cor:thmgenmax}
Let $T$ be a continuous operator on a separable $F$-algebra $X$. Assume that, for any pair $(U,V)$ of non-empty open sets in $X$, for any open neighbourhood $W$ of zero, 
and for any positive integer $m$, one can find $u\in U$ and an integer $N$ such that $T^N(u^n)\in W$ for all $n<m$ and $T^N(u^m)\in V$. Then $T$ admits a hypercyclic algebra.
\end{corollary}
\begin{proof}
It is straightforward to show that the assumptions of Theorem \ref{thm:generalcriterion} with $d=1$ and $\beta_A=\max A$ are satisfied.
\end{proof}

In this work, we will often use the inverse choice for $\beta_A$.
\begin{corollary}\label{cor:thmgenmin}
Let $T$ be a continuous operator on a separable $F$-algebra $X$. Assume that, for any pair $(U,V)$ of non-empty open sets in $X$, for any open neighbourhood $W$ of zero, 
and for any positive integers $m_0<m_1$, one can find $u\in U$ and an integer $N$ such that $T^N(u^m)\in W$ for all $m\in\{m_0+1,\dots,m_1\}$ and $T^N(u^{m_0})\in V$. Then $T$ admits a hypercyclic algebra.
\end{corollary}
\begin{proof}
This is now Theorem \ref{thm:generalcriterion} with $d=1$ and $\beta_A=\min A$.
\end{proof}

\subsection[Algebras not contained in a finitely generated algebra]{Countably generated, free hypercyclic algebras.}

The conclusion of Theorem \ref{thm:generalcriterion} about infinitely generated hypercyclic algebras does not prevent the possibility for such an algebra to be contained in a finitely generated 
algebra. Furthermore, it is well known that every at least two generated algebra contains an infinitely generated subalgebra. 
For all of the examples of this paper we may avoid this scenario thanks to the following result. We notice that due to \cite[Corollary 2.7]{BP18}, a countably generated free algebra is not contained in a finitely generated one.

\begin{corollary}\label{generators}
 Let $X$ be a separable commutative $F$-algebra that contains a dense freely generated subalgebra.
 Let $T$ be a continuous operator on  $X$ and let $d\geq 1$. Assume that for any $A\subset \NN_0^d \backslash\{(0,\dots,0)\}$ finite and non-empty, 
   for any non-empty open subsets $U_1,\dots,U_d,V$ of $X$, for any neighbourhood $W$ of $0$, 
 there exist $u=(u_1,\dots,u_d)\in U_1\times\cdots\times U_d$, $\beta\in A$  and $N\geq 1$ such that $T^N(u^\beta)\in V$ and $T^N(u^\alpha)\in W$ for all $\alpha\in A$, $\alpha\neq\beta$.
 Then $T$ admits a $d$-generated, free hypercyclic algebra. Moreover, if the assumptions are satisfied for all $d\geq 1$, then $T$ admits a dense, countably generated, free hypercyclic algebra.
\end{corollary}

\begin{proof} 
By Theorem~\ref{thm:generalcriterion} the set of $d$-tuples generating a hypercyclic algebra for $T$ is residual in $X^d$.
By \cite[Proposition 2.4]{BP18}, the set of $u$ in $X^d$ that induce a $d$-generated, free algebra is residual in $X^d$. For the conclusion we just need to pick an element in the intersection of those two sets.

For the second claim, we consider $X^{\NN}$ endowed with the product topology. By the assumption, for each $N\in \NN$, the set
$$
H_N=\{ (u_n)\in X^{\NN}: (u_1,\dots ,u_N) \,\, \mbox{generates a hypercyclic algebra for} \,\, T\}
$$
is residual in $X^{\NN}$ and hence, by the Baire category theorem, the set $H=\bigcap_{N=1}^{\infty}H_N$ is residual as well. The algebra generated by any $u\in H$ is hypercyclic for $T$.
Furthermore, by \cite[Proposition 2.4]{BP18}, the set of sequences of $X^{\NN}$ which generate a dense and free algebra is also residual. The conclusion follows by one more application of the Baire category theorem.
\end{proof}

Hence, we need to provide for our examples a dense and freely generated subalgebra. This is quite easy for a Fréchet sequence algebra endowed with the Cauchy product: provided $\textrm{span}(e_n)$ is dense, it always contain a dense and 
freely generated subalgebra, namely the unital algebra generated by the sequence $e_1$. This covers the case of $H(\CC)$ and of the disc algebra. This is slightly more difficult for a Fréchet sequence algebra endowed 
with the coordinatewise product.

\begin{lemma}\label{free}
Let $X$ be a Fr\'echet sequence algebra, endowed with the pointwise product, and for which the sequence $(e_n)$ is a Schauder basis. 
Then $X$ has a dense freely generated subalgebra.
\end{lemma}

\begin{proof}
Let $(b_n)\subset (0,1)$ be such that the series $\sum_{n=0}^{\infty}b_n \|e_n\|_n$ converges.
Consider the sequence of natural numbers $(a_n)$ such that $a_0=0$ and $a_n=a_{n-1}+n, n\in \NN$ and define $c_n:=\min_{l\in [a_{m-1},a_m)}b_l$, if $n\in [a_{m-1},a_m)$.
The series $\sum_{n=0}^{\infty}c_ne_n$ converges absolutely since $c_n\leq b_n,\ n\in \NN$.

Define a sequence $(\lambda_n)\subset (0,1)$ inductively as follows: choose $\lambda_0 \in (0,1)$ and, 
for $n\in \NN,$ take $\lambda_n \in (0,1)\setminus \{ \lambda_0^{p_0}\dots \lambda_{n-1}^{p_{n-1}}: p_0,\dots ,p_{n-1}\in \QQ \}$. 
Observe that if $\lambda=(\lambda_0,\dots ,\lambda_p),$ $p\in \NN _0$ and $\alpha \neq \beta \in \NN _0^p$, then $\lambda ^{\alpha}\neq \lambda ^{\beta}$. Let now, for each $n\in \NN _0$,
$$
g_n=e_n+\sum_{k=1}^{\infty}\lambda_n^{n+k}c_{n+k}e_{n+k},
$$
where the convergence of the series is ensured by the convergence of $\sum_{n=0}^{+\infty}b_n\|e_n\|_n$ and the inequality $\lambda_n^{n+k}c_{n+k}<b_{n+k}$.
We claim that the algebra generated by the $g_n,\ n\in \NN $ is dense and free. 

First, we show that the sequence $\{g_n: n\in \NN \}$ is algebraically independent. For that reason, let
\begin{equation} \label{1}
\sum_{\alpha \in A}a_{\alpha}g^{\alpha}=0,
\end{equation}
where $A=\{ \alpha(1),\dots ,\alpha(q) \}\subset \NN_0^p$, $g=(g_1,\dots ,g_p)\in X^p$, and $p\in \NN$. If we consider the coordinate $N=n+k$ in equation \eqref{1}, we get the following equation which holds for all $N$ sufficiently large.
\begin{equation} \label{2}
\sum_{i=1}^qa_{\alpha(i)}(\lambda_1^Nc_N)^{\alpha_1(i)}\dots (\lambda_p^Nc_N)^{\alpha_p(i)}=0.
\end{equation}
Now we may choose $N$ sufficiently big such that $N, \dots ,N+q-1\in [a_{m-1},a_m)$  for some $m$, which means that $c_N=\dots =c_{N+q-1}=b_m$. Equation \eqref{2} then becomes
$$
\sum_{i=1}^qa_{\alpha(i)}(\lambda_1^Mb_m)^{\alpha_1(i)}\dots (\lambda_p^Mb_m)^{\alpha_p(i)}=0,
$$
where $M$ ranges over $N,\dots ,N+q-1$. Setting $A$ the matrix
$$
\begin{bmatrix}
(\lambda_1^Nb_m)^{\alpha_1(1)}\cdots (\lambda_p^Nb_m)^{\alpha_p(1)} &\dots &(\lambda_1^Nb_m)^{\alpha_1(q)}\cdots (\lambda_p^Nb_m)^{\alpha_p(q)} \\
\vdots &\ddots &\vdots \\
(\lambda_1^{N+q-1}b_m)^{\alpha_1(1)}\cdots (\lambda_p^{N+q-1}b_m)^{\alpha_p(1)} &\dots &(\lambda_1^{N+q-1}b_m)^{\alpha_1(q)}\cdots (\lambda_p^{N+q-1}b_m)^{\alpha_p(q)} 
\end{bmatrix},
$$
we find the matrix equality 
$$A\begin{bmatrix}
a_{\alpha(1)}\\
\vdots \\
a_{\alpha(q)}
\end{bmatrix}
=0.
$$
The determinant of the square matrix $A$, after making use of the Vandermonde identity, is
$$
\prod_{i=1}^q \prod_{j=1}^p (b_m\lambda_j^N)^{\alpha_j(i)}\prod_{i>j}[(\lambda_1^{\alpha_1(i)}\dots \lambda_p^{\alpha_p(i)})-(\lambda_1^{\alpha_1(j)}\dots \lambda_p^{\alpha_p(j)})]\neq 0.
$$
Hence, we get that $a_{\alpha(i)}=0$ for all $i=1,\dots ,q$.

Next, we show that the algebra generated by $\{g_n: n\in \NN \}$ is dense in $X$. We will show that the elements $e_n, n\in \NN$, are in the closure of this algebra. Let us fix $n\in\NN$ and observe that, for all $p\in\NN$, 
$$
g_n^p-e_n=\sum_{k=1}^{\infty}\lambda_n^{(n+k)p}c_{n+k}^p e_{n+k}.
$$
Fix $q\in\NN$ and let $\veps>0$. There exists $N\geq q$ such that, for all $p\in\NN$, 
$$ 
\left \| \sum_{k>N}\lambda_n^{(n+k)p}c_{n+k}^pe_{n+k}\right \|_q\leq \sum_{k>N}b_{n+k}\|e_{n+k}\|_q <\varepsilon.
$$
Since furthermore
$$
\sum_{k=1}^N\lambda_n^{(n+k)p}c_{n+k}^p e_{n+k}\xrightarrow[p\rightarrow \infty]{}  0,
$$
we conclude that $g_n^p\xrightarrow[p\rightarrow \infty]{}e_n$.
\end{proof}

We conclude this subsection by comparing Corollary \ref{generators} with \cite[Remark 3.4]{BP18}. Corollary \ref{generators} allows the index $\beta$ to depend on $A, U_1,\dots ,U_d,V$ and $W$ providing, at least theoretically, an extra flexibility and range of application for the result. Practically, throughout the paper, $\beta$ will depend only on $A$ in which case Corollary \ref{generators} and \cite[Remark 3.4]{BP18} coincide. We were unable to find an example where Corollary \ref{generators} applies while \cite[Remark 3.4]{BP18} does not.


\section{Convolution operators with $|\phi(0)|>1$}

\subsection{Operators with many eigenvectors}

\label{sec:convolution}

In this section we shall deduce Theorem \ref{thm:convintro} from a more general assertion on operators having many eigenvalues. As Theorem \ref{thm:generalcriterion} does for Corollary \ref{cor:thmgenmax}, this generalized statement also includes \cite[Theorem 2.1]{Bayhcalg} as a particular case. Before stating and proving it, let us add some notation.

Given $p,d\in\NN$, we denote each set $\{1,...,p\}$ by $I_p$, each $d$-tuple $(j_1,...,j_d)\in I_p^d$ by the multi-index $\textbf{j}\in I_p^d$ and each product $a_{j_1}\cdots a_{j_m}$ by the symbol $a_{\textbf{j}}$. We allow $d=0$ with the convention that, in this case, $a_{\textbf{j}}=1$.

\begin{theorem}\label{versionconvol}
Let $X$ be an $F$-algebra and let $T\in\mathcal{L}(X)$. Assume that there exist a function $E:\CC\to X$ and an entire function $\phi:\CC\to\CC$ satisfying the following assumptions:
\begin{enumerate}[(a)]
  \item for all $\lambda\in\CC$, $TE(\lambda)=\phi(\lambda)E(\lambda)$;
  \item for all $\lambda, \mu\in\CC$, $E(\lambda)E(\mu)=E(\lambda+\mu)$;
  \item for all $\Lambda \subset\CC$ with an accumulation point, the linear span of $\{E(\lambda):\lambda\in\Lambda\}$ is dense in $X$;
  \item $\phi$ is not a multiple of an exponential function;
  \item for all $I\in\mathcal{P}_f(\NN)\backslash\{\varnothing\}$, there exist $m\in I$ and $a,b\in\CC$ such that $|\phi(mb)|>1$ and, for all $n\in I$ and $d\in\{0,...,n\}$, with $(n,d)\neq(m,m)$, $|\phi(db+(n-d)a)|<|\phi(mb)|^{d/m}.$
\end{enumerate}
Then $T$ supports a hypercyclic algebra.
\end{theorem}

The proof of this result follows the lines of that of Theorem 2.1 in \cite{Bayhcalg}, replacing Corollary \ref{cor:thmgenmax} by the more general Theorem \ref{thm:generalcriterion}. For the sake of completeness, we include the details.

\begin{proof}
Let $(U,V, W)$ be a triple of non-empty open sets in $X$, with $0\in W$, and let $I\in\mathcal{P}_f(\NN)\backslash\{\varnothing\}$. By the hypothesis there are $m\in I$ and $a,b\in\C$ satisfying (e).
Define $w_0:=mb$ and let $\delta>0$ be small enough and $w_1,w_2\in B(w_0,\delta)$ so that 
\begin{enumerate}[(i)]                                       
\item $|\phi|>1$ on $B(w_0,\delta)$;
\item $t\mapsto \log|\phi(tw_1+(1-t)w_2)|$ is strictly convex (the existence of $w_1,w_2\in B(w_0,\delta)$ comes from \cite[Lemma 2.2]{Bayhcalg} and is a consequence of (d));
\item for all $n\in I$ and $d\in\{0,...,n\}$, with $(n,d)\neq(m,m)$, and for all $\lambda_1,...,\lambda_d\in B(w_0,\delta)$ and $\gamma_1,...,\gamma_{n-d}\in B(a,\delta)$, 
\begin{equation}\label{starcond0} \biggl|\phi\biggl(\frac{\lambda_1+\cdots+\lambda_d}{m}+\gamma_1+\cdots+\gamma_{n-d}\biggr)\biggr|<\left(|\phi(\lambda_1)|\times\cdots\times |\phi(\lambda_d)|\right)^{1/m}.
\end{equation}
\end{enumerate}
The last condition can be satisfied because \[ \frac{\lambda_1+\cdots+\lambda_d}{m}+\gamma_1+\cdots+\gamma_{n-d}=db+(n-d)a+z,\] where the size of $z$ can be controlled through $\delta$. Now, since $B(a,\delta)$ and $[w_1,w_2]$ have accumulation points, we can find $p,q\in\NN$, $a_1,...,a_p,b_1,...,b_q\in\C$, $\gamma_1,...,\gamma_p\in B(a,\delta)$ and $\lambda_1,...,\lambda_q\in[w_1,w_2]$ with \[\sum_{l=1}^pa_lE(\gamma_l)\in U~~\text{and}~~\sum_{j=1}^qb_jE(\lambda_j)\in V.\] For some big $N\in\NN$ (which will be determined later in the proof) and each $j\in\{1,...,q\}$, let $c_j:=c_j(N)$ be any complex number satisfying \[c_j^{m}(N)=\frac{b_j}{\phi(\lambda_j)^N}\] and define \[u:=u(N)=\sum_{l=1}^pa_lE(\gamma_l)+\sum_{j=1}^qc_jE\biggl(\frac{\lambda_j}{m}\biggr).\] For the powers of $u$ we have the formula
\begin{equation}\label{formulaun}
u^n=\sum_{d=0}^n\sum_{\substack{\textbf{l}\in I_p^{n-d}\\ \textbf{j}\in I_q^d}}\alpha(\textbf{l},\textbf{j},d,n)a_{\textbf{l}}c_{\textbf{j}}E\biggl(\gamma_{l_1}+\cdots+\gamma_{l_{n-d}}+ \frac{\lambda_{j_1}+\cdots+\lambda_{j_d}}{m}\biggl).
\end{equation}
We claim that, if $N$ is taken large enough, $u=u(N)$ satisfies the conditions of the general criterion with $d=1$, what will complete the proof.

That $u\in U$ for large $N$ is clear since, from (i), $c_j(N)\to 0$ as $N\to\infty$. Applying $T^N$ to $u^n$ we see that we need to study the behaviour (as $N$ grows) of
\begin{equation}\label{study0}
c_{\textbf{j}}(N)\biggl[\phi\biggl(\gamma_{l_1}+\cdots+\gamma_{l_{n-d}}+\frac{\lambda_{j_1}+\cdots+\lambda_{j_d}}{m}\biggl)\biggl]^N.
\end{equation}
For $n\in I\backslash\{m\}$ we have that (\ref{study0}) goes to 0 as $N$ grows by the inequality (\ref{starcond0}) and the definition of $c_j$, $j=1,...,q$. 
This way we get $T^N(u^n)\in W$ for all $n\in I\backslash\{m\}$ if $N$ is large enough. Now let us consider the case $n=m$. We have
\begin{align*}
  u^{m} &= \sum_{d=0}^{m-1}\sum_{\substack{\textbf{l}\in I_p^{m-d}\\ \textbf{j}\in I_q^d}} \alpha(\textbf{l},\textbf{j},d,m)a_{\textbf{l}}c_{\textbf{j}} E\biggl(\gamma_{l_1}+\cdots+\gamma_{l_{m-d}}+\frac{\lambda_{j_1}+\cdots+\lambda_{j_d}}{m}\biggl) \\
                    &+ \sum_{\textbf{j}\in I_q^m\backslash D_q} \alpha(\textbf{j},m)c_{\textbf{j}}E\biggl(\frac{\lambda_{j_1}+\cdots+\lambda_{j_m}}{m}\biggl)\\
                    &+\sum_{j=1}^qc_j^{m}E(\lambda_j)\\
                    &=:v_1+v_2+v_3,
\end{align*}
where
\[
v_1:=\sum_{d=0}^{m-1}\sum_{\substack{\textbf{l}\in I_p^{m-d}\\ \textbf{j}\in I_q^d}} \alpha(\textbf{l},\textbf{j},d,m)a_{\textbf{l}}c_{\textbf{j}} 
E\biggl(\gamma_{l_1}+\cdots+\gamma_{l_{m-d}}+\frac{\lambda_{j_1}+\cdots+\lambda_{j_d}}{m}\biggl),
\]
\[v_2:=\sum_{\substack{\textbf{j}\in I_q^{m}}\backslash D_q}\alpha(\textbf{j},m)c_{\textbf{j}}E\biggl(\frac{\lambda_{j_1}+\cdots+\lambda_{j_{m}}}{m}\biggl),~~~~v_3:=\sum_{j=1}^qc_j^mE(\lambda_j)
\]
and $D_q$ is the diagonal of $I_q^{m}$, that is, the set of all $m$-tuples $(j,...,j)$ with $1\leq j\leq q$. Again we have $T^N(v_1)\to 0$ as $N\to\infty$ from (\ref{starcond0}).
Furthermore, since $t\in[0,1]\mapsto\log|\phi(tw_1+(1-t)w_2)|$ is strictly convex, we have
\[
\biggl|\phi\biggl(\frac{\lambda_{j_1}+\cdots+\lambda_{j_{m}}}{m}\biggl)\biggl|<|\phi(\lambda_{j_1})|^{1/m}\cdots|\phi(\lambda_{j_{m}})|^{1/m}.
\]
From this we conclude that
$$\begin{array}{l}
\displaystyle |c_{\textbf{j}}(N)|\cdot\Biggl|\phi\Biggl(\frac{\lambda_{j_1}+\cdots+\lambda_{j_{m}}}{m}\Biggl)\Biggr|^N \\
\displaystyle \quad\quad = \Biggl|\frac{b_{\textbf{j}}^{1/m}}{|\phi(\lambda_{j_1})|^{N/m}\cdots|\phi(\lambda_{j_n})|^{N/m}}\Biggr|\cdot \Biggl|\phi\Biggl(\frac{\lambda_{j_1}+\cdots+\lambda_{j_{m}}}{m}\Biggr)\Biggr|^N\\
\displaystyle \quad\quad =\bigl|b_{\textbf{j}}^{1/m}\bigl|\cdot\left|\frac{\phi\Bigl(\frac{\lambda_{j_1}+\cdots+\lambda_{j_{m}}}{m}\Bigr)}{|\phi(\lambda_{j_1})|^{1/m}\cdots|\phi(\lambda_{j_{m}})|^{1/m}}\right|^N\to0~~\text{as}~~N\to\infty,
\end{array}
$$
what shows that $T^N(v_2)$ also tends to 0 as $N\to\infty$. Finally, by the definition of $c_j$, $j=1,...,q$, we get
\[
T^Nv_3=\sum_{j=1}^qb_jE(\lambda_j)\in V
\]
for all $N\in\NN$. This completes the proof.
\end{proof}

We now deduce a more readable corollary,  when the entire function $\phi$ is ``well behaved'' in some half-line of the complex plane (like in the Figure \ref{figure} for example).

\begin{corollary}\label{wellbehaved}
Let $X$ be an $F$-algebra and let $T\in\mathcal{L}(X)$. Assume that there exist a function $E:\C\to X$ and an entire function $\phi:\C\to\C$ satisfying the following assumptions:
\begin{enumerate}[(a)]
  \item for all $\lambda\in\C$, $TE(\lambda)=\phi(\lambda)E(\lambda)$;
  \item for all $\lambda, \mu\in\C$, $E(\lambda)E(\mu)=E(\lambda+\mu)$;
  \item for all $\Lambda \subset\C$ with an accumulation point, the linear span of $\{E(\lambda):\lambda\in\Lambda\}$ is dense in $X$;
  \item $\phi$ is not a multiple of an exponential function;
  \item there exist $v\in\C$ and a real number $p>0$ such that $|\phi(v)|>1$ and $|\phi(tv)|\leq1$ for all $t>p$.
\end{enumerate}
Then $T$ supports a hypercyclic algebra.
\end{corollary}

\begin{figure}[ht]
    \centering
    \includegraphics[height=4.35cm]{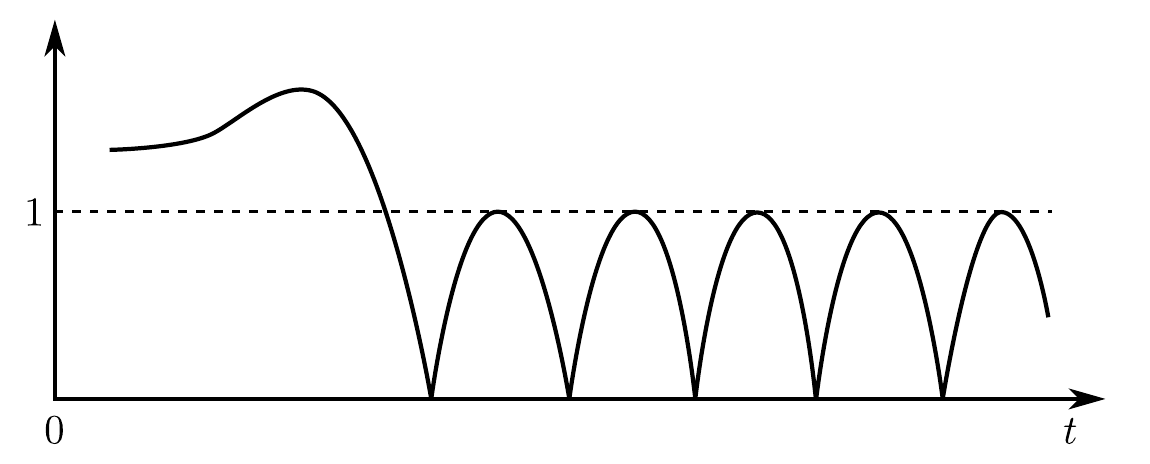}
    \caption{Graph of $t\mapsto |\phi(tv)|$}
    \label{figure}
\end{figure}

For the proof of Corollary \ref{wellbehaved}, we are going to need the following simple lemma.

\begin{lemma}\label{lemmawellbehaved}
Let $\phi$ be entire and $\Lambda\subset\mathbb R$ have an accumulation point in it. If $|\phi(t)|=1$ for all $t\in\Lambda$ then $|\phi(t)|=1$ for all $t\in\mathbb R$.
\end{lemma}
\begin{proof}
Since $|\phi(t)|=1$ for all $t\in\Lambda\subset\R$, we have $\overline{\phi(\overline{t})}\cdot\phi(t)=\overline{\phi(t)}\cdot\phi(t)=1$, hence
\[
\overline{\phi(\overline{t})}=\phi(t)^{-1}, \forall t\in\Lambda\subset\R.
\]
Since this is a holomorphic equality ($t\mapsto\overline{\phi(\overline{t})}$ is entire), it extends to the whole complex plane. In particular it holds for all $t\in\R$, that is,
\[
\overline{\phi(t)}=\overline{\phi(\overline{t})}=\phi(t)^{-1}, \forall t\in\R,
\]
thus
\[
|\phi(t)|=1, \forall t\in\R,
\]
as we wanted.
\end{proof}

\begin{proof}[Proof of Corollary \ref{wellbehaved}] We may assume without loss of generality that $v=1$.
Let $t_0>0$ be the smallest positive real number such that $|\phi(t)|\leq1$ for all $t\geq t_0$.
We just need to prove that condition (e) of Theorem \ref{versionconvol} is satisfied. 
So let $I\in\mathcal{P}_f(\NN)\backslash\{\varnothing\}$ be arbitrary and set $m=\min I$.
We can find $t_1<t_0$ near enough to $t_0$ so that $|\phi(t_1)|>1$ and $t_1+\frac{t_1}{m}>t_0$.
 Letting $b:=\frac{t_1}{m}$ we have $|\phi(mb)|>1$. Now fix $a_0>t_0$ and take $\epsilon\in(0,1/2)$ such that $a_0-\epsilon>t_0$. 
 There exists $a\in[a_0-\epsilon,a_0+\epsilon]$ such that $|\phi(db+(n-d)a)|<1$ for all $n\in I$ and $d\in\{0,...,n\}$ with $(n,d)\neq(m,m)$. In fact, if this is not the case then, 
 for each $a\in[a_0-\epsilon,a_0+\epsilon]$, we can find a point $t_a=d_ab+(n_a-d_a)a$, with $n_a\in I$, $d_a\in\{0,...,n_a\}$ and $(n_a,d_a)\neq(m,m)$, such that $|\phi(t_a)|\geq 1$.
 Since $m=\min(I)$, $b=t_1/m$ and $a>t_0$, we get $t_a>t_0$ so that $|\phi(t_a)|\leq 1$.
 This way, varying $a\in[a_0-\epsilon,a_0+\epsilon]$ we find infinitely many points $t_a$ within $[t_0,\max I(a_0+\veps)]$ 
in which $|\phi|$ assumes the value 1. 
The set $\Lambda$ composed by these points is infinite, closed and subset of the compact $[t_0,\max (I)(\alpha_0+\epsilon)]$, hence $\Lambda$ has an accumulation point. 
By Lemma \ref{lemmawellbehaved} we conclude that $|\phi(t)|=1$ for all $t\in\R$, which contradicts the fact that $|\phi(t_1)|>1$. This completes the proof. 
\end{proof}

\subsection{Applications to convolution operators}

We now observe that we may apply Corollary \ref{wellbehaved} to convolution operators $\phi(D)$ with $|\phi(0)|>1$, where $E(\lambda)(z)=e^{\lambda z}$. This yields immediately Theorem \ref{thm:convintro}. 
We may also apply Corollary \ref{wellbehaved} to handle the case $|\phi(0)|=1$.

\begin{corollary}
Let $P\in\CC[z]$ be a non-constant polynomial and let $\phi(z)=P(z)e^z$. Then $\phi(D)$ supports a hypercyclic algebra.
\end{corollary}
\begin{proof}
The case $|P(0)|<1$ is done in \cite{Bayhcalg}, the case $|P(0)|>1$ is settled by Theorem \ref{thm:convintro}. It remains to consider the case $|P(0)|=1$. Since $|P(it)|$ tends to $+\infty$ as $t$ tends to $+\infty$, there exists $t_0\in\RR$ such that $|\phi(it_0)|>1$. By continuity of $|\phi|$, there exists $v=|v|e^{i\theta}$ with $\theta\in (\pi/2,3\pi/2)$ such that $|\phi(v)|>1$. Now, because $v$ lies in the left half-plane, $|\phi(tv)|$ tends to $0$ as $t$ tends to $+\infty$. We may conclude with Corollary \ref{wellbehaved}.
\end{proof}

We finish this section by pointing out that Theorem \ref{versionconvol} can also handle functions which do not satisfy the properties described above.

\begin{example}
The convolution operator induced by $\phi(z)=\frac 12e^z+e^{iz}-\frac 14$ supports a hypercylic algebra (let us observe that $|\phi(0)|>1$ and that $\phi$ does not tend to $0$ along any ray). 
Indeed, for any $I\in\mathcal P_f(\mathbb N)\backslash\{\varnothing\}$ we choose
$m=\max(I)$ and take $a=k(2\pi i)$ and $b=k2\pi$ for some large integer $k$. 
Let $n\in I$ and $d\in\{0,\dots,n\}$ with $(n,d)\neq (m,m)$. Then 
$$\phi(db+(n-d)a)=\frac 12 e^{2dk\pi}+e^{-2(n-d)k\pi}-\frac 14.$$
In particular, 
$$|\phi(mb)|=\frac 12e^{2mk\pi}+\frac 34>1.$$
When $d=0$, 
$$|\phi(na)|=\left|e^{-2nk\pi}+\frac 14\right|<1.$$
Finally, 
$$\left|\phi(db+(n-d)a)\right|\leq \frac12 e^{2dk\pi}+\frac 34$$
and we have, for all $d=1,\dots,m-1$, 
$$\left(\frac 12e^{2dk\pi}+\frac 34\right)^m\leq \left(\frac 12e^{2mk\pi}+\frac 34\right)^d$$
if $k$ is large enough.
\end{example}

\begin{remark}
Combining the previous arguments with that of \cite[Section 6]{Bayhcalg}, under the assumptions of Theorem \ref{thm:convintro}, $\phi(D)$ admits
a dense, countably generated, free hypercyclic algebra.
\end{remark}



\section{Weighted shifts on Fréchet sequence algebras}

\label{sec:ws}


\subsection{Fr\'echet sequence algebras with the coordinatewise product}

We begin with the proof of Theorem \ref{thm:mainws}. We first explain where the property of admitting a continuous norm comes into play.

\begin{lemma}
Let $X$ be a Fr\'echet sequence algebra for the coordinatewise product and with a continuous norm. Then the sequence $(e_n)$ is bounded below.
\end{lemma}
\begin{proof}
Let $q\geq 1$ be such that $\|\cdot\|_q$ is a norm on $X$. Then for all $n\in\NN_0$, 
$$0<\|e_n\|_q=\|e_n\cdot e_n\|_q\leq \|e_n\|_q^2$$
which shows that $\|e_n\|_q\geq 1$.
\end{proof}

We shall prove the following precised version of Theorem \ref{thm:mainws}.
\begin{theorem}\label{thm:wsprecised}
Let $X$ be a Fr\'echet sequence algebra for the coordinatewise product and with a continuous norm. Assume that $(e_n)$ is a Schauder basis for $X$. Let also $B_w$ be a bounded weighted shift on $X$. The following assumptions are equivalent.
\begin{enumerate}[(i)]
\item $B_w$ supports a dense, countably generated, free hypercyclic algebra. 
\item $B_w$ supports a hypercyclic algebra.
\item For all $m\geq 1$, there exists $x\in X$ such that $x^m$ is a hypercyclic vector for $B_w$.
\item For all $m\geq 1$, for all $L\in\NN$, there exists a sequence of integers $(n_k)$ such that, for all $l=0,\dots,L$, $((w_{l+1}\cdots w_{n_k+l})^{-1/m}e_{n_k+l})$ tends to zero.
\item There exists a sequence of integers $(n_k)$ such that, for all $\gamma>0$ and for all $l\in\mathbb N$,  $\big((w_{l+1}\cdots w_{n_k+l})^{-\gamma}e_{n_k+l}\big)$ tends to zero.
\end{enumerate}
\end{theorem}

\begin{proof}

The implications $(i)\implies (ii)$ and $(ii)\implies (iii)$ are trivial. The proof of $(iii)\implies (iv)$ mimics that of the necessary condition for hypercyclicity. Let $m\geq 1$ and $x\in X$ be 
such that $x^m\in HC(B_w)$. Write $x=\sum_{n=0}^{+\infty}x_n e_n$. Since $(e_n)$ is a Schauder basis, the sequence $(x_n e_n)$ goes to zero.
Moreover, there exists a sequence of integers $(n_k)$ such that $(B_w^{n_k}(x^m))_k$ goes to $e_0+\cdots+e_L$. Since convergence in $X$ implies coordinatewise convergence,
for all $l=0,\dots,L$, $(w_{l+1}\cdots w_{n_k+l}x_{n_k+l}^m)$ converges to $1$. Hence the sequences $\big((w_{l+1}\cdots w_{n_k+l})^{1/m}x_{n_k+l}\big)$ are bounded below. Writing 
$$(w_{l+1}\cdots w_{n_k+l})^{-1/m}e_{n_k+l}=\frac{1}{(w_{l+1}\cdots w_{n_k+l})^{1/m}x_{n_k+l}}\cdot x_{n_k+l}e_{n_k+l}$$
we get the result.

To prove that $(iv)\implies (v)$, observe that a diagonal argument ensure the existence of a sequence $(n_k)$ such that, for all $m\geq 1$ and all $l\in\NN$, 
the sequence $((w_{l+1}\cdots w_{n_k+l})^{-1/m}e_{n_k+l})$ tends to zero. Now we can conclude by observing that, since the sequence $(e_n)$ is bounded below, if $\big((w_{l+1}\cdots w_{n_k+l})^{-1/m}e_{n_k+l}\big)$ 
tends to zero for some $m$, then $(w_{l+1}\cdots w_{n_k+l})$ tends to $+\infty$ and, in particular, $\big((w_{l+1}\cdots w_{n_k+l})^{-\gamma}e_{n_k+l}\big)$ tends to zero for all $\gamma\geq 1/m$.  

It remains to prove the most difficult implication, $(v)\implies (i)$. We start by fixing a sequence of integers $(n_k)$ such that for all $\gamma>0$ and for all $l\in\mathbb N$,  $\big((w_{l+1}\cdots w_{n_k+l})^{-\gamma}e_{n_k+l}\big)$ goes to zero. We intend to apply Theorem \ref{thm:generalcriterion}. Thus, let $d\geq 1$ and $A\subset\NN_0^d \backslash\{(0,\dots,0)\}$ 
be finite and 
non-empty. For $\alpha\in A$ we define the linear form $L_\alpha$ on $\mathbb R^d$ by $L_\alpha(\kappa)=\sum_{j=1}^d \alpha_j\kappa_j$.
Since $L_\alpha$ and $L_{\alpha'}$ coincide only on a hyperplane for $\alpha\neq \alpha'$,
there exist $\kappa\in(0,+\infty)^d$ and $\beta=\beta_A\in A$ such that $0<L_{\beta}(\kappa)<L_\alpha(\kappa)$ for all $\alpha\neq\beta$, $\alpha\in A$. 
Without loss of generality, we may assume that $L_{\beta}(\kappa)=1$.

Let now $U_1,\dots,U_d,V$ be non-empty open subsets of $X$ and let $W$ be a neighbourhood of zero. Let $x_1,\dots,x_d$ belonging respectively to $U_1,\dots,U_d$ with finite support and let 
$y=\sum_{l=0}^p y_l e_l$ belonging to $V$. We set, for $j=1,\dots,d$,
$$u_j:=u_j(n_k)=x_j+\sum_{l=0}^p \frac{y_l^{\kappa_j}}{(w_{l+1}\cdots w_{n_k+l})^{\kappa_j}}e_{n_k+l}.$$
Our assumption implies that, provided $n_k$ is large enough, $u_j$ belongs to $U_j$ for all $j=1,...,d$. Moreover, again if $n_k$ is large enough (larger than the size of the support of each $x_j$), for all $\alpha\in A$, 
\[ B_w^{n_k}(u^\alpha)=\sum_{l=0}^p \frac{y_l^{L_\alpha(\kappa)}}{(w_{l+1}\cdots w_{n_k+l})^{L_\alpha(\kappa)-1}}e_{l}. \]
In particular, for $\alpha=\beta$, $B_w^{n_k}(u^{\beta})=y\in V$ whereas, for $\alpha\neq\beta$, since $L_\alpha(\kappa)-1>0$ and since the sequences $(w_{l+1}\cdots w_{n_k+l})$ tend to $+\infty$,
we get $B_w^{n_k}(u^\alpha)\in W$ provided $n_k$ is large enough. Hence, $B_w$ admits a dense and not finitely generated hypercyclic algebra. 
\end{proof}


As recalled in the introduction, the hypercyclicity of $B_w$ on $X$ is equivalent to the existence of a sequence $(n_k)$ such that, for all $l\in\NN$, 
$((w_{l+1}\cdots w_{n_k+l})^{-1}e_{n_k+l})$ tends to zero. It is well known that this last condition is equivalent to the following one, which seems much weaker:
there exists a sequence of integers $(n_k)$ such that $((w_1\cdots w_{n_k})^{-1}e_{n_k})$ tends to zero. In view of this and of Theorem \ref{thm:wsprecised}, it is tempting to conjecture that $B_w$ supports
a hypercyclic algebra if and only if there exists a sequence of integers $(n_k)$ such that, for all $\gamma>0$, $( (w_1\cdots w_{n_k})^{-\gamma}e_{n_k})$ tends to zero. Unfortunately,
this is not the case, as the following example points out.

\begin{example}
 Let $X=\{(x_n)\in\omega: |x_n|a_n\to 0\}$ where $a_{2n}=1$ and $a_{2n+1}=2^n$ endowed with $\|x\|=\sup_n |x_n|a_n$ and let $w$   be the weight such that  $w_1\cdots w_{2n}=2^{n-1}$
 and $w_1\cdots w_{2n+1}=2^{2n}$. Then $w$ is an admissible weight on $X$, $( (w_1\cdots w_{2n})^{-\gamma}e_{2n})$ tends to zero for all $\gamma>0$ but $B_w$ does not admit a hypercyclic
 algebra.
\end{example}

\begin{proof}
 We first observe that, endowed with the coordinatewise product, $X$ is a Fréchet sequence algebra (since $a_n\geq 1$ for all $n$). To prove that $w$ is admissible, it suffices to observe
 that $w_k \|e_{k-1}\|\leq 2 \|e_k\|$ for all $k$. The construction of $w$ ensures that $w_{2n}=2^{-(n-1)}$ and $w_{2n+1}=2^{n+1}$. Hence the previous inequality is clearly satisfied
 if we separate the case $k$ even and $k$ odd. Moreover for all $\gamma>0$, 
 \[ (w_1\cdots w_{2n})^{-\gamma} \|e_{2n}\|=2^{-\gamma(n-1)}\xrightarrow{n\to+\infty}0. \]
 To prove that $B_w$ does not support a hypercyclic algebra, it suffices to observe that, for all $n\geq 1$, $(w_1\cdots w_{2n+1})^{-1/2}\|e_{2n+1}\|=1$, which implies that condition (v) of Theorem
 \ref{thm:wsprecised} cannot be satisfied.
\end{proof}

Nevertheless, if we add an extra assumption on $X$, then we get the expected result.

\begin{corollary}
 Let $X$ be a Fr\'echet sequence algebra for the coordinatewise product and with a continuous norm. 
 Assume that $(e_n)$ is a Schauder basis for $X$. Assume also that, for all admissible weights $w$, for all $\gamma>0$, $w^\gamma$ is admissible.
 Let $B_w$ be a bounded weighted shift on $X$. The following assumptions are equivalent.
\begin{enumerate}[(i)]
\item $B_w$ supports a hypercyclic algebra.
\item For all $\gamma>0$, there exists a sequence $(n_k)$ such that $((w_1\cdots w_{n_k})^{-\gamma} e_{n_k})$ tends to zero.
\end{enumerate}
\end{corollary}

\begin{proof}
 We assume that (ii) is satisfied and we show that, for all $\gamma>0$ and for all $L\in\NN$, there exists a sequence $(m_k)$ such that
 $( (w_1\cdots w_{m_k+l})^{-\gamma}e_{m_k+l})$ tends to zero. An application of Theorem \ref{thm:wsprecised} will then allow to conclude. 
 It is easy to get this sequence $m_k$. Indeed, it is sufficient to set $m_k=n_k-L$, since in that case
 \[ (w_1\cdots w_{m_k+l})^{-\gamma}e_{m_k+l}=(B_{w^\gamma})^{L-l}\left((w_1\cdots w_{n_k})^{-\gamma}e_{n_k}\right)\]
 which goes to zero by continuity of $B_{w^\gamma}$. 
\end{proof}

We may observe that our favorite sequence spaces (namely unweighted $\ell_p$-spaces or $H(\CC)$) satisfy the assumptions of the last corollary. We may also observe that on unweighted $\ell_p$-spaces as well as on any Fréchet sequence algebra with a continuous norm such that $(e_n)$ is bounded, the convergence
of $\big((w_1\cdots w_{n_k})^{-\gamma}e_{n_k}\big)$ to zero is equivalent to the convergence of $(w_1\cdots w_{n_k})$ to $+\infty$. Hence, we may formulate the following corollary.

\begin{corollary}\label{cor:wslp}
 Let $X$ be a Fr\'echet sequence algebra for the coordinatewise product and with a continuous norm. 
 Assume that $(e_n)$ is a Schauder basis for $X$ and that $(e_n)$ is bounded. Assume also that, for all admissible weights $w$, for all $\gamma>0$, $w^\gamma$ is admissible.
 Let $B_w$ be a bounded weighted shift on $X$. The following assumptions are equivalent.
\begin{enumerate}[(i)]
\item $B_w$ supports a hypercyclic algebra.
\item $B_w$ is hypercyclic.
\item There exists a sequence $(n_k)$ such that $(w_1\cdots w_{n_k})$ tends to $+\infty$.
\end{enumerate}
\end{corollary}

\begin{remark}
On $H(\CC)$, the sequence $(z^n)$ is unbounded. Nevertheless, any hypercyclic weighted shift $B_w$ on $H(\CC)$ supports a hypercyclic algebra. Indeed, for a sequence of integers $(n_k)$, 
\begin{align*}
&(w_1\cdots w_{n_k})^{-1}z^{n_k}\textrm{ tends to $0$ in }H(\CC)\\
\iff&\forall q\geq 1,\ (w_1\cdots w_{n_k})^{-1}q^{n_k}\textrm{ tends to $0$}\\
\iff&\forall q\geq 1,\ \forall \gamma>0,\ (w_1\cdots w_{n_k})^{-\gamma}q^{n_k}\textrm{ tends to $0$}\\
\iff&\forall \gamma>0,\ (w_1\cdots w_{n_k})^{-\gamma}z^{n_k}\textrm{ tends to $0$ in }H(\CC).
\end{align*}
\end{remark}

\begin{remark}
Theorem \ref{thm:wsprecised}  points out one difficulty when dealing with hypercyclic algebras: to admit a hypercyclic algebra is not a property preserved by similarity. 
Indeed, let $X=\{(x_n)\in\omega : |x_n|2^n\to 0\}$ endowed with $\|x\|=\sup_n |x_n|2^n$ and let $w$ be the weight such that  $w_1\cdots w_n=n\cdot 2^n$ for all $n\geq 1$. Then $(w_1\cdots w_n)^{-1}2^n$ goes to zero whereas $(w_1\cdots w_n)^{-1/2}2^n$ tends to $+\infty$, showing that $B_w$ is hypercyclic but that no square vector $x^2$ belongs to $HC(B_w)$. 

Let now $(\rho_n)$ be defined by $\rho_1=1$ and $\rho_n=n/(n-1)$ for $n\geq 2$. Then $B_w$ acting on $X$ is similar to $B_\rho$ acting on $c_0$, the similarity being given by $S:X\to c_0,\ (x_n)\mapsto (2^n x_n)$. 
But $B_\rho$ admits a hypercyclic algebra, which is not the case of $B_w$. Of course, the problem is that $S$ is not a morphism of algebra.
\end{remark}

When $X$ does not admit a continuous norm, one cannot apply Theorem \ref{thm:wsprecised}.
The space $\omega$ is the prototypal example of a Fr\'echet space without a continuous norm (in fact, by a result of Bessaga and Pelczinski \cite{BesPel57}, a Fréchet space fails to admit a continous 
norm if and only if it has a subspace isomorphic to $\omega$) and we shall now concentrate on this space. On $\omega$, for all weight sequences $w=(w_n)$, the weighted shift $B_w$ is bounded, hypercyclic and satisfies $(iv)$. If the sequences  $(w_l\cdots w_{n+l})$ converge to $+\infty$ for all $\ell\geq 0$, then an easy modification of the proof of the previous theorem shows that $B_w$ admits a hypercyclic algebra. On the other hand, if the sequences  $(w_l\cdots w_{n+l})$ converge to $0$ for all $\ell\geq 0$, then we may modify the previous proof using Corollary \ref{cor:thmgenmax} instead of Corollary \ref{cor:thmgenmin} to prove that $B_w$ still admits a hypercyclic algebra. A completely different case is that of the unweighted shift $B$.
It is a hypercyclic multiplicative operator on $\omega$. By \cite[Theorem 16]{BCP18}, $B$ supports a hypercyclic algebra if and only if for each nonconstant polynomial $P\in\mathbb C[X]$ with $P(0)=0$,
the map $\tilde P:\omega\to\omega,\ x\mapsto P(x)$ has dense range. This is clearly true.

We now show that every weighted shift on $\omega$ admits a hypercyclic algebra showing that, coordinate by coordinate, $B_w$ behaves like one of the three previous models.
\begin{theorem}\label{thm:algebraomega}
Every weighted shift $B_w$ on $\omega$ endowed with the coordinatewise product supports a hypercyclic algebra.
\end{theorem}
\begin{proof}
For $V$ a non-empty open subset of $\omega$, $I\subset\NN$ finite and non-empty and $s>0$, let us define
\begin{align*}
E(I,s)=\Big\{P\in\mathbb C[z]:\ & |\hat P(\min I)|\geq 1/s,\ |\hat P(\max I)|\geq 1/s,\\
&|\hat P(n)|\leq s\textrm{ for all }n\in\mathbb N,\\
&\hat P(n)=0\textrm{ when }n\notin I\Big\}
\end{align*}
\begin{align*}
\mathcal A(I,s,V)=\Big\{u\in \omega:\ \forall P\in E(I,s),\ \exists N\geq 1,\ T^N(P(u))\in V\Big\}.
\end{align*}
As in the proof of Theorem \ref{thm:generalcriterion}, it is enough to prove that each set $\mathcal A(I,s,V)$ is dense and open. The last property follows easily from the compactness of $E(I,s)$. 
Thus, let us fix $I,s$ and $V$ and let us prove that $\mathcal A(I,s,V)$ is dense. We set $m_0=\min(I)$ and $m_1=\max(I)$. Let $U$ be a non-empty open subset of $\omega$.
Let $p\in\NN_0$, $u_0,\dots,u_p,v_0,\dots,v_p\in \CC$ and $\veps>0$ be such that, for all $x,y\in\omega$, 
\[ |x_l-u_l|<\veps\textrm{ for all }l=0,\dots,p\textrm{ implies }x\in U, \]
\[ |y_l-v_l|<\veps\textrm{ for all }l=0,\dots,p\textrm{ implies }y\in V.\]
Let us first look at the sequence $(w_1\cdots w_n)$. Three possibilities (which are not mutually exclusive) can occur:
\begin{itemize}
\item either $(w_1\cdots w_n)$ is bounded and bounded below;
\item or it admits a subsequence going to zero;
\item or it admits a subsequence going to $+\infty$.
\end{itemize}
Thus, we get the existence of a subsequence $(w_1\cdots w_{n_k})$ going to $a_0\in [0,+\infty]$. We then do the same with $(w_2\cdots w_{n_k+1})$ and so on. By successive extractions, we get the existence of a sequence of integers $(n_k)$ (we can assume that $n_{k+1}-n_k>p$ for all $k$ and that $n_0>p$) and of $a_0,\dots,a_p\in [0,+\infty]$ such that, for all $l=0,\dots,p$, $(w_{l+1}\cdots w_{n_k+l})$ tends to $a_l$. 
We set $A_1=\{l\in\{0,\dots,p\}:\ a_l=+\infty\}$, $A_2=\{l\in\{0,\dots,p\}:\ a_l=0\}$ and $A_3=\{l\in\{0,\dots,p\}:\ a_l\in(0,+\infty)\}$.

We fix now $(\alpha(k))$, $(\beta(k))$ two sequences of non-zero complex numbers and $(z(k))$ a sequence in $\mathbb C^{p+1}$ such that $(\alpha(k),\beta(k),z(k))$ is dense in $\mathbb C^{p+3}$. We set 
$$x=u+\sum_{k=0}^{+\infty}y(k)$$
where, for $l=0,\dots,p$,
$$y_{n_k+l}(k)=
\left\{
\begin{array}{cl}
\displaystyle \frac{v_l^{1/m_0}}{\alpha(k)^{1/m_0}(w_{l+1}\cdots w_{n_k+l})^{1/m_0}}&\textrm{provided }l\in A_1,\\[0.5cm]
\displaystyle \frac{v_l^{1/m_1}}{\beta(k)^{1/m_1}(w_{l+1}\cdots w_{n_k+l})^{1/m_1}}&\textrm{provided }l\in A_2,\\[0.5cm]
z_l(k)&\textrm{ provided }l\in A_3
\end{array}
\right.$$
and $y_i(k)=0$ if $i\neq n_k,\dots,n_k+p$.

We claim that $x\in U\cap \mathcal A(I,s,V)$. The definition of $\veps$ and $p$ ensure that $x\in U$. Let $P\in E(I,s)$. There exists an increasing function $\phi:\NN\to\NN$
such that $\alpha(\phi(k))\to \hat P(m_0)$, $\beta(\phi(k))\to \hat P(m_1)$ and $a_l P(z_l(\phi(k)))\to v_l$ for all $l\in A_3$. We claim that $(B_w^{n_{\phi(k)}}(P(x)))$ belongs to $V$ provided $k$ is large enough. It suffices to prove that for $l=0,\dots,p$, the $l$-th coordinate of $B_w^{n_{\phi(k)}}(P(x))$ tends to $v_l$. Assume first that $l\in A_1$. This $l$-th coordinate is equal to 
$$w_{l+1}\cdots w_{n_{\phi(k)}+l}P\left(\frac{v_l^{1/m_0}}{\alpha(\phi(k))^{1/m_0}(w_{l+1}\cdots w_{n_{\phi(k)}+l})^{1/m_0}}\right).$$
Now, since $w_{l+1}\cdots w_{n_{\phi(k)}+l}$ tends to $+\infty$, and $m_0=\min(I)$,
$$w_{l+1}\cdots w_{n_{\phi(k)}+l}P\left(\frac{v_l^{1/m_0}}{\alpha(\phi(k))^{1/m_0}(w_{l+1}\cdots w_{n_{\phi(k)}+l})^{1/m_0}}\right)=\hat P(m_0)\frac{v_l}{\alpha(\phi(k))}+o(1)$$
and this tends to $v_l$. When $l\in A_2$, the proof is similar since now, because $w_{l+1}\cdots w_{n_{\phi(k)}+l}$ tends to 0, and $m_1=\max(I)$,
$$w_{l+1}\cdots w_{n_{\phi(k)}+l}P\left(\frac{v_l^{1/m_1}}{\beta(\phi(k))^{1/m_1}(w_{l+1}\cdots w_{n_{\phi(k)}+l})^{1/m_1}}\right)=\hat P(m_1)\frac{v_l}{\beta(\phi(k))}+o(1)$$
and this also goes to $v_l$. Finally, when $l\in A_3$, the $l$-th coordinate of $B_w^{n_{\phi(k)}}(P(x))$ is equal to $w_{l+1}\cdots w_{n_{\phi(k)}+l}P(z_l(\phi(k)))$ which tends again to $v_l$.
\end{proof}

Theorem \ref{thm:algebraomega} has an analogue (with a completely different proof!) if we endow $\omega$ with the Cauchy product: see \cite[Corollary 3.9]{FalGre18}. We also point out that the existence
of a continuous norm is an important assumption in several problems in linear dynamics, for instance for the
existence of a closed infinite dimensional subspace of hypercyclic vectors (see \cite{menet}).

\subsection[Bilateral shifts]{Bilateral shifts on Fréchet sequence algebras with the coordinatewise product}


In this section, we investigate the case of bilateral shifts on a Fréchet sequence algebra $X$ on $\ZZ$; namely, $X$ is a subset of $\CC^\ZZ$ endowed with the coordinatewise product under which it is an $F$-algebra. We intend to give an analogue of Theorem \ref{thm:wsprecised} for bilateral shifts on $X$. The statement and the methods are close to what happens for unilateral shifts. Since we do not want to give an exhaustive list of examples in this work, there is an extra interest for looking at bilateral shifts: a small subtility appears in this case, since the condition that appears is not symmetric for the positive part of the weight and for the negative one. This will lead us to an interesting example of a hypercyclic operator $T$ supporting a hypercyclic algebra such that $T^{-1}$ does not. 

\begin{theorem}\label{thm:bilateralshifts}
Let $X$ be a Fréchet sequence algebra on $\ZZ$ for the coordinatewise product, with a continuous norm. Assume that $(e_n)_{n\in\ZZ}$ is a Schauder basis for $X$. Let also $B_w$ be a bounded bilateral shift on $X$ such that, for all $\gamma\in(0,1)$, $B_{w^\gamma}$ is bounded. The following assertions are equivalent.
\begin{enumerate}[(i)]
\item $B_w$ supports a hypercyclic algebra.
\item For all $m\geq 1$, for all $L\in\NN$, there exists a sequence of integers $(n_k)$ such that, for all $l=-L,\dots,L$, $\big( (w_{l+1}\cdots w_{n_k+l})^{-1/m}e_{n_k+l}\big)$ and $\big( w_{l}\cdots w_{-n_k+l+1}e_{-n_k+l}\big)$ tend to zero.
\end{enumerate}
\end{theorem}
\begin{proof}
$(ii)\implies (i)$. We intend to apply Corollary \ref{cor:thmgenmin}. Let $1\leq m_0<m_1$, let $U,V$ be nonempty open subsets of $X$ and let $W$ be a neighbourhood of zero. Let $x,y$ belonging to $U$ and $V$ respectively, with finite support contained in $[-p,p]$. Write $y=\sum_{l=-p}^p y_le_l$ and let $(n_k)$ be the sequence given in (ii) for $m=m_0$ and $L=p$. Define
$$u:=u(n_k)=x+\sum_{l=-p}^p \frac{y_l^{1/m_0}}{(w_{l+1}\cdots w_{n_k+l})^{1/m_0}}e_{n_k+l}.$$
Provided $k$ is large enough, $u$ belongs to $U$. Moreover, for $m\in\{m_0,\dots,m_1\}$, 
$$B_w^{n_k}(u^m)=\sum_{l=-p}^p w_l\cdots w_{-n_k+l+1}x_l^m e_{-n_k+l}+
\sum_{l=-p}^p \frac{y_l^{m/m_0}}{(w_{l+1}\cdots w_{n_k+l})^{\frac m{m_0}-1}}e_l.$$
For all values of $m$, it is clear that 
$$\sum_{l=-p}^p w_l\cdots w_{-n_k+l+1}x_l^me_{-n_k+l}\xrightarrow{k\to+\infty}0.$$
Hence, for $m=m_0$ and provided $k$ is large enough, $B_w^{n_k}(u^{m_0})$ belongs to $V$. 
Furthermore, if $m>m_0$, since each sequence $(w_{l+1}\cdots w_{n_k+l})^{-1}$ goes to zero (recall that $(e_n)$ is bounded below), then $B_w^{n_k}(u^m)$ belongs to $W$ for large values of $k$, showing that $B_w$ admits a hypercyclic algebra. 

$(i)\implies (ii)$. The proof is slightly more difficult than for unilateral shifts. Fix $m$ and $L$ and let $x\in X$ be such that $x^m\in HC(B_w)$. Let $(s_k)$ be an increasing sequence of integers such that $B_w^{s_k}(x^m)$ tends to $e_{-L}+\cdots+e_L$. We fix some $s\in\NN$ (which can be taken equal to some $s_{k_0}$) such that, for all $l=-L,\dots,L$, the $l$-th coordinate of $B_w^s(x)$ is not equal to zero. We then consider $y\in X$ defined by $y_l=(w_{l+1}\cdots w_{l+s})^{1/m}x_{l+s}$ (namely, $y=B_{w^{1/m}}^s(x)$) and we set $n_k=s_k-s$. It is easy to check that $B_w^{n_k}(y^m)=B_w^{s_k}(x^m)$. Hence, it goes to $e_{-L}+\cdots +e_L$. This implies that 
\begin{itemize}
\item for all $l=-L,\cdots, L$, 
$$w_l\cdots w_{-n_k+l+1}y_l^m e_{-n_k+l}\textrm{ tends to }0.$$
\item for all $l=-L,\cdots,L$, 
$$w_{l+1}\cdots w_{n_k+l}y_{n_k+l}^m\textrm{ tends to }1.$$
\end{itemize}
We conclude as in the unilateral case, using that $y_l$ is never equal to zero for $l=-L,\cdots,L$. 
\end{proof}
We can then state corollaries similar to what happens in the unilateral case.

\begin{corollary}
 Let $X$ be a Fr\'echet sequence algebra on $\ZZ$ for the coordinatewise product and with a continuous norm. 
 Assume that $(e_n)_{n\in\ZZ}$ is a Schauder basis for $X$. Assume also that, for all admissible weights $w$, for all $\gamma\in (0,1)$, $w^\gamma$ is admissible.
 Let $B_w$ be a bounded bilateral weighted shift on $X$. The following assumptions are equivalent.
\begin{enumerate}[(i)]
\item $B_w$ supports a hypercyclic algebra.
\item For all $\gamma>0$, there exists a sequence $(n_k)$ such that $((w_1\cdots w_{n_k})^{-\gamma} e_{n_k})$ tends to zero and $(w_{-1}\cdots w_{-n_k}e_{-n_k})$ tends to $0$.
\end{enumerate}
\end{corollary}

\begin{corollary}
 Let $X$ be a Fr\'echet sequence algebra on $\ZZ$ for the coordinatewise product and with a continuous norm. 
 Assume that $(e_n)$ is a Schauder basis for $X$ and that $(e_n)$ is bounded. Assume also that, for all admissible weights $w$, for all $\gamma\in(0,1)$, $w^\gamma$ is admissible.
 Let $B_w$ be a bounded bilateral weighted shift on $X$. The following assumptions are equivalent.
\begin{enumerate}[(i)]
\item $B_w$ supports a hypercyclic algebra.
\item $B_w$ is hypercyclic.
\item There exists a sequence $(n_k)$ such that $(w_1\cdots w_{n_k})$ and $(w_{-1}\cdots w_{-n_k})$ tend to $+\infty$.
\end{enumerate}
\end{corollary}

On the contrary, the nonsymmetry of the conditions in (ii) of Theorem \ref{thm:bilateralshifts} proves to be useful to get the following example.
\begin{example}\label{ex:inverse}
There exists an invertible operator $T$ on a Banach algebra such that $T$ supports a hypercyclic algebra and $T^{-1}$ does not.
\end{example}

\begin{proof}
Let 
$$X=\left\{x\in\CC^\ZZ:\ x_n (|n|+1)\xrightarrow{n\to\pm \infty}0\right\},$$ 
endowed with 
$$\|x\|=\sup_n |x_n| (|n|+1).$$
Equipped with the coordinatewise product, $X$ is a  Fréchet sequence algebra. Let $w$ be the weight defined by $w_0=1$, $w_n=2$ and $w_{-n}=n^2/(n+1)^2$ for $n>0$. For all $\gamma>0$, the weighted shift $B_{w^\gamma}$ is bounded on $X$. Moreover, it satisfies the assumptions of Theorem \ref{thm:bilateralshifts} with $(n_k)$ equal to the whole sequence of integers. In particular, $w_{-1}\cdots w_{-n} \|e_{-n}\|=(n+1)^{-1}$ tends to zero.

It is plain that $B_w$ is invertible and that its inverse is the forward shift $F_\rho$, defined by $F_\rho(e_n)=\rho_{n+1}e_{n+1}$ with $\rho_n=1/w_n$. Assume that $F_\rho$ supports a hypercyclic algebra. Then we apply the symmetrized version of  Theorem \ref{thm:bilateralshifts} adapted to forward shifts with $m=2$ to get the existence of a sequence $(n_k)$ such that $(\rho_{-1}\cdots \rho_{-n_k})^{-1/2}e_{-n_k}$ tends to zero. 
This is impossible since 
$$\left\|(\rho_{-1}\cdots \rho_{-n_k})^{-1/2}e_{-n_k}\right\|\sim_{k\to+\infty}\frac{n_k}{n_k^{2\times 1/2}}=1.$$
\end{proof}

\color{black}


\subsection{Fréchet sequence algebras for the convolution product}

This subsection is devoted to the proof of Theorem \ref{thm:wscauchy}. We first have to give the meaning of a regular Fréchet sequence algebra. 
Let $(X,(\|\cdot\|_q))$ be such a Fréchet sequence algebra for the Cauchy product. We will say that $X$ is \emph{regular} provided that it satisfies the following three properties:
\begin{enumerate}[(a)]
 \item $X$ admits a continuous norm;
 \item $(e_n)$ is a Schauder basis of $X$;
 \item for any $r\geq 1$, there exists $q\geq 1$ and $C>0$ such that, for all $n,k\geq 0$,
 \[ \|e_n\|_r\cdot \|e_k\|_r\leq C\|e_{n+k}\|_q.\]
\end{enumerate}
Let us make some comments on these assumptions. Conditions (a) and (b) are standard in this work. We shall use (a) by assuming that $\|e_n\|_q>0$ for all $n\in\NN_0$
and all $q>0$. Regarding (c), it should be thought as a reverse inequality for the continuity of the product in $X$. Observe also that $H(\CC)$ and $\ell_1$
are clearly regular. However, this is not the case of all Fréchet sequence algebras for the Cauchy product. 
Pick for instance any sequence $(a_n)$ of positive real numbers such that, for all $n,p,q\in\NN_0$, with $n=p+q$, $a_n\leq a_p\cdot a_q$
and $a_n^2/a_{2n}\to +\infty$ (the sequence $a_n=1/n!$ does the job). Then the Banach space $X=\{x\in\omega:\|x\|_X=\sum_{n\geq 0} a_n |x_n|<+\infty\}$
is a Fréchet sequence algebra for the Cauchy product which does satisfy (c).

A consequence of (c) is the following technical lemma which will be crucial later.
\begin{lemma}\label{lem:wscauchy}
 Let $X$ be a regular Fréchet sequence algebra for the Cauchy product and let $(w_n)$ be an admissible weight sequence on $X$. Then, for all $M\geq 1$,
 for all $r\geq 1$, and for all $\rho \geq 0$, there exist $C>0$ and $q\geq r$ such that, for all $n\geq M$, for all $u<v$ in $\{n-M,\dots,n\}$, for all $k\in\{n-M+\rho,\dots,n+\rho \}^{v-u}$, 
 \begin{equation}\label{eq:wscauchy} \prod_{j=1}^{v-u} w_{k_j} \|e_u\|_r\leq C \|e_v\|_q. 
 \end{equation}

\end{lemma}

Before to proceed with the proof, let us comment the statement of Lemma \ref{lem:wscauchy}. The inequality \eqref{eq:wscauchy}
is nothing else than the continuity of $B_w$ if we assume that $k_j=u+j$ for $j=1,\dots,v-u$. The regularity of $X$ (and more precisely the third condition) will imply that we may slightly move
the indices $k_j$.

\begin{proof}
 Let us fix $M\geq 1$ and observe that $v-u$ may only take the values $1,\dots,M$. Then, upon doing a finite induction and taking suprema, we need only to prove that,
 for all $r\geq 1$, and for all $\rho \geq 0$, there exist $C>0$ and $q\geq r$ such that, for all $n\geq M$, for all $u\in\{n-M,\dots,n-1\}$, for all $k\in \{n-M+\rho,\dots,n+\rho \}$, 
 \begin{equation}\label{eq:lemwscauchy}
  w_k \|e_u\|_r\leq C \|e_{u+1}\|_q,
 \end{equation}
 a property which should be thought as a strong version of the continuity of $B_w$. Assume first that $u\geq k-1$. Then, writing $e_u=e_{k-1}\cdot e_{u-(k-1)}$
 and using the continuity of the product and of $B_w$, we get the existence of $C>0$ and $q_1\geq r$ such that
 \begin{align*}
  w_k \|e_u\|_r&\leq w_k \|e_{k-1}\|_r \cdot \|e_{u-(k-1)}\|_r \\
  &\leq C_1  \|e_k\|_{q_1} \cdot \|e_{u-(k-1)}\|_{q_1}.
 \end{align*}
 We now use property (c) for $r=q_1$ to deduce the existence of $C_2>0$ and $q_2\geq q_1$ such that 
 \[ w_k \|e_u\|_r \leq C_1C_2 \|e_{u+1}\|_{q_2}. \]
 Hence, \eqref{eq:lemwscauchy} is proved for $q=q_2$ and $C=C_1C_2$.
 If we assume that $u<k-1$, then the argument is completely similar by exchanging the place where we use the continuity of the product and property (c). Precisely,
 \begin{align*}
  w_k \|e_u\|_r&=w_k\frac{\|e_u\|_r\cdot \|e_{(k-1)-u}\|_r}{\|e_{(k-1)-u}\|_r}\\
  &\leq C_1 w_k \frac{\| e_{k-1} \|_{q_1}}{\|e_{(k-1)-u}\|_r}\\
  &\leq C_1 C_2 \frac{\|e_k\|_{q_2}}{\|e_{k-(u+1)}\|_r}\\
  &\leq C_1 C_2 \frac{\|e_{k-(u+1)}\|_{q_2}}{\|e_{k-(u+1)}\|_r}\|e_{u+1}\|_{q_2}. 
 \end{align*}
 Hence, \eqref{eq:lemwscauchy} is proved for $q=q_2$ and 
 \[ C=\max \left\{C_1C_2 \frac{\|e_l\|_{q_2}}{\|e_l\|_{r}}:1\leq l\leq M+\rho -1 \right\}. \]
\end{proof}

Lemma \ref{lem:wscauchy} will be used through the following more particular form.
\begin{corollary}\label{cor:wscauchy}
 Let $X$ be a regular Fréchet sequence algebra for the Cauchy product and let $(w_n)$ be an admissible weight sequence on $X$. Then, for all $m\geq 1$, for all $N\geq 1$, for
 all $r\geq 1$, and for all $\rho \geq 0$, there exist $C>0$ and $q\geq 1$ such that, for all $n\geq mN$, for all $s\in\{1,\dots,N\}$, 
\[ (w_{n-s+1+\rho })^{m-1}\cdots (w_{n-1+\rho })^{m-1}(w_{n+\rho} )^{m-1} \|e_{n-ms+m\rho}\|_r \leq C\|e_{n-s+\rho}\|_q. \]
\end{corollary}
\begin{proof}
We apply the previous lemma with $M=mN$ to get $q'\geq r$ and $C'>0$ such that
 \[(w_{n-s+1+\rho })^{m-1}\cdots (w_{n-1+\rho })^{m-1}(w_{n+\rho} )^{m-1} \|e_{n-ms}\|_r \leq {C'}\|e_{n-s}\|_{q'}.\]
Now, using property (c) and the continuity of the product on a Fréchet algebra, we get $q\geq q'\geq r$ and $C''>0$ with
\begin{align*}
&(w_{n-s+1+\rho })^{m-1}\cdots (w_{n-1+\rho })^{m-1}(w_{n+\rho} )^{m-1} \|e_{n-ms+m\rho}\|_r \\ &\quad\quad\quad\quad\leq (w_{n-s+1+\rho })^{m-1}\cdots (w_{n-1+\rho })^{m-1}(w_{n+\rho} )^{m-1} \|e_{n-ms}\|_r\|e_\rho\|_r\|e_{(m-1)\rho}\|_r \\ &\quad\quad\quad\quad\leq C'\|e_{n-s}\|_{q'}\|e_\rho\|_{q'}\|e_{(m-1)\rho}\|_r \\ &\quad\quad\quad\quad\leq C'C''\|e_{n-s+\rho}\|_q\|e_{(m-1)\rho}\|_r\\&\quad\quad\quad\quad = C\|e_{n-s+\rho}\|_q,
\end{align*}
where $C=C'C''\|e_{(m-1)\rho}\|_r$.
\end{proof}

Before proceeding with the proof of Theorem \ref{thm:wscauchy}, let us make some comments on the difference between the coordinatewise product and the convolution product. Let $P(z)=\sum_{m\in I}\hat P(m)z^m$ be a nonzero polynomial, $x,y$ with finite support. The work done in Section \ref{sec:criterion} shows that it is important for us to find $u$ close to $x$ and $N$ such that $B_w^N(P(u))$ is close to $y$. In both cases, $u$ will be of the form $u=x+z$, where $z$ has finite support and $\min(\supp(z))\gg \max(\supp(x))$. For the coordinatewise product, each $u^m$ has the same support as $u$. Moreover, since $z$ has to be small, $z^m$ becomes smaller as $m$ increases. Hence, in $P(u)$, the most important term was $u^{m_0}$, where $m_0=\min(I)$ (we assume $\hat P(m)\neq 0\iff m\in I$) and it was natural to apply Corollary \ref{cor:thmgenmin}. 

Regarding the convolution product, the support of $u^m$ is now moving to the right: $\max(\supp(u^{m+1}))\geq \max(\supp(u^m))$. By using a specific translation term (an idea coming from \cite{shkarin} and \cite{FalGre19}), we will arrange the choice of $N$ such that 
$B_w^N(u^m)=0$ when $m<m_1:=\max(I)$ and $B_w^N(u^{m_1})$ is close to $y$. This explains why we will rather use Corollary \ref{cor:thmgenmax}.

\begin{proof}[Proof of Theorem \ref{thm:wscauchy}]
We start with a hypercyclic weighted shift $B_w$ and prove that $B_w$ supports a dense hypercyclic algebra which is not contained in a finitely generated algebra. Let $d\geq1$, $A\in\mathcal P_f(\NN_0^d)\backslash\{\varnothing\}$ and $U_1,...,U_d,$ $V,W\subset X$ be open and non-empty, with $0\in W$. We choose $\beta=\max A$ under the lexicographical order, say $\beta=(m,\beta_2,...,\beta_d)$. Upon interchanging the coordinates in $\NN^d$, we may and will assume that $m>0$. Let $x_1,...,x_d$ belonging respectively to $U_1,...,U_d$ and let $y=\sum_{l=0}^py_le_l$ belonging to $V$. We can find $r\geq1$, $\delta>0$ and a ball $B\subset W$ for the seminorm $\|\cdot\|_r$ and with radius $\delta$ such that $y+B\subset V$ and $x_i+B\subset U_i$, for all $i=1,...,d$. 


Since $\beta>\alpha$ for all $\alpha\in A\backslash\{\beta\}$ under the lexicographical order, 
we may find integers $s_1,\dots,s_d$ with $s_i>4p$ such that
\begin{equation}\label{cond:s}
(m-\alpha_1)s_1+(\beta_2-\alpha_2)s_2+\cdots+(\beta_d-\alpha_d)s_d>3p,\text{  for all  }\alpha\in A\setminus \{ \beta \}.
\end{equation}
The procedure to do this is the following. First find $s_d$ such that $(\beta _d-\alpha _d)s_d>3p$, for all $\alpha _d<\beta _d$, $\alpha_d \in \pi _d(A)$. Next, find $s_{d-1}$ such that $(\beta _{d-1}-\alpha_{d-1})s_{d-1}+(\beta_d-\alpha_d)s_d>3p$, for all $\alpha_{d-1}<\beta_{d-1}\in \pi_{d-1}(A)$ and $\alpha_d \in \pi_d(A)$. Continuing inductively, after finitely many steps we define $s_2$ such that $(\beta_2-\alpha_2)s_2+\dots +(\beta_d-\alpha_d)s_d>3p$, for all $\alpha_2<\beta_2$, $\alpha_2\in \pi_2(A)$, and $\alpha_i\in \pi_i(A)$, for $i=3,\dots ,d$. Finally we chose $s_1$ large enough so that $(m-\alpha_1)s_1+(\beta_2-\alpha_2)s_2+\cdots+(\beta_d-\alpha_d)s_d>3p$, for all $\alpha_1<m$, $\alpha_1 \in \pi_1(A)$, $\alpha_i \in \pi_i(A)$, for $i=2,\dots ,p$. This way we get (\ref{cond:s}).

These $s_i$ being fixed, we now choose positive real numbers $\eta_2,\dots,\eta_d$ such that
\begin{equation}\label{cond:eta}
\|\eta_ie_{s_i}\|_r<\delta,\text{  for all  }i=2,...,d.
\end{equation}



We will distinguish two cases in order to apply Corollary \ref{generators}. The most difficult one is $m>1$, an assumption that we now make.  We set  $\rho=\beta_2s_2+\cdots+\beta_ds_d$ and we consider a sequence $(J_k)$ going to $+\infty$ such that, for all $l=0,\dots,p$, 
\[ (w_1\cdots w_{mJ_k-3p+l+\rho})^{-1}e_{mJ_k-3p+l+\rho}\xrightarrow{k\to+\infty}0. \]
Indeed,  let $(m_k)$ be a sequence of integers such that $(w_1\cdots w_{m_k})^{-1}e_{m_k}$ goes to zero and $m_k\geq m+\rho$ for all $k$. Define $J_k$ as the single integer such that $m_k -m < m J_k-2p+\rho\leq m_k$. Then, for all $l=0,\dots,p$, 
\[ (w_1\cdots w_{mJ_k-3p+\rho+l})^{-1} e_{mJ_k-3p+\rho+l}=B_w^{ m_k-mJ_k+3p-\rho-l}\big( (w_1\cdots w_{m_k})^{-1} e_{m_k}\big) \] which tends to zero by the continuity of $B_w$ and because 
$$0\leq m_k-mJ_k+3p-\rho-l\leq m+p.$$ 

We now proceed with the construction of the vectors $u_1,\dots,u_d$ required to apply Corollary \ref{generators}. We set, for $k$ large enough, $N=mJ_k-3p+\rho$ and
\begin{gather*}
\veps=\max_{0\leq l\leq p}\left(\frac{\|e_{J_k-3p+l}\|_r}{w_1\cdots w_{mJ_k-3p+l+\rho}}\right)^{\frac{1}{2(m-1)}}\times \min\left(\frac{1}{\|e_{J_k}\|_r},\frac{1}{(w_1\cdots w_{mJ_k+\rho})^{1/m}}\right)^{\frac{1}{2}},\\
d_j=\frac{w_1\cdots w_jy_j}{\eta_2^{\beta_2}\cdots \eta_d^{\beta_d}m\varepsilon^{m-1}w_1\cdots w_{mJ_k-3p+j+\rho}}.
\end{gather*}
We also define
\begin{gather*}
u_1=x_1+\sum_{j=0}^pd_je_{J_k-3p+j}+\varepsilon e_{J_k},\\
u_i=x_i+\eta_ie_{s_i},\text{  for  }i=2,...,d.
\end{gather*}
Let us postpone the proof of the following facts.
\begin{gather}
\label{condconv:1} \varepsilon \|e_{J_k}\|_r\to0\text{  as  }k\to+\infty, \\
\label{condconv:2} |d_j|\cdot \|e_{J_k-3p+j+\rho}\|_r\to0\text{  as  }k\to+\infty,\text{  for all  }j=0,...,p, \\
\label{condconv:3} \varepsilon^mw_1\cdots w_{mJ_k+\rho}\to0\text{  as  }k\to+\infty.
\end{gather}
From (\ref{condconv:1}) and (\ref{condconv:2}) we get $u_1\in U_1$ if $k$ is large enough and from (\ref{cond:eta}) we get $u_i\in U$ for $i=2,...,d$.
We claim that $u^\alpha\in\ker B_w^N$ for all $\alpha\in A$ with $\alpha\neq\beta$. In fact, for a given $\alpha\in A\backslash\{\beta\}$, say $\alpha=(\alpha_1, ..., \alpha_d)$, we have
\begin{gather*}
\max \left(\supp(u^\alpha)\right) \leq \alpha_1J_k+\alpha_2s_2+\cdots+\alpha_ds_d,
\end{gather*}
so the claim follows by (\ref{cond:s}) since $\alpha_1\leq m$ and for $k$ large enough, $J_k\geq s_1$.
Finally, for the main power $\beta$ we write
$$u^\beta=z+\sum_{j=0}^p \eta_2^{\beta_2}\cdots \eta_d^{\beta_d}m\veps^{m-1}d_j e_{mJ_k-3p+j+\rho}+\veps^m \eta_2^{\beta_2}\cdots \eta_d^{\beta_d}e_{mJ_k+\rho},$$
where the maximum of the support of $z$ is less than $N=mJ_k-3p+\rho$. Indeed, a term in $z$ can come
\begin{itemize}
\item either from a term in $u_1^{\beta_1}$ with support in $[0,mJ_k-4p]$ so that the maximum of the support of this term is at most $mJ_k-4p+\rho<N$;
\item or from a term in some $u_i^{\beta_i}$, $i=2,\dots,d$, with support in $[0,(\beta_i-1)s_i+p]$. The maximum of the support of such a term is then at most $mJ_k+\rho+p-s_i<N$ since $s_i>4p$.
\end{itemize}
Hence, we get 
\begin{gather*}
B_w^Nu^\beta=y+\frac{\varepsilon^m \eta_2^{\beta_2}\cdots \eta_d^{\beta_d} w_1\cdots w_{mJ_k+\rho}}{w_1\cdots w_{3p}}e_{3p},
\end{gather*}
which belongs to $V$ by (\ref{condconv:3}) if $k$ is big enough. It remains now to show that properties (\ref{condconv:1}), (\ref{condconv:2}) and (\ref{condconv:3}) hold true.

\smallskip

Let us first prove \eqref{condconv:1}. By property (c) and an easy induction, there exist $q\geq 1$ and $C>0$ (depending on $r$ and $m$) such that, for all $k\geq 1$ and all $l\in\{0,\dots,p\}$,
\begin{align*}
\frac{\|e_{J_k-3p+l}\|_r^{\frac{1}{2(m-1)}}\cdot \|e_{J_k}\|_r}{\|e_{J_k}\|_r^{\frac 12}}
&= \frac{\left(\|e_{J_k-3p+l}\|_r\cdot \|e_{J_k}\|_r^{m-1}\cdot \|e_\rho\|_r\right)^{\frac 1{2(m-1)}}}
{ \|e_\rho\|_r^{\frac 1{2(m-1)}}}\\
 &\leq C \|e_{mJ_k-3p+l+\rho}\|_q^{\frac 1{2(m-1)}}.
\end{align*} 
Hence,
\[ \veps \|e_{J_k}\|_r \leq C \max_{0\leq l\leq p} \left(\frac { \|e_{mJ_k-3p+l+\rho}\|_q}{w_1\cdots w_{mJ_k-3p+l+\rho}}\right)^{\frac 1{2(m-1)}}   \]
and this goes to zero as $k$ tends to $+\infty$.

Regarding \eqref{condconv:2}, we first write 
\begin{align*}
|d_j|\cdot \|e_{J_k-3p+j+\rho}\|_r &\leq C \frac{\|e_{J_k-3p+j+\rho}\|_r}{w_1\cdots w_{mJ_k-3p+j+\rho}}\times \min_{0\leq l\leq p} \left(\frac{w_1\cdots w_{mJ_k-3p+l+\rho}}{\|e_{J_k-3p+l}\|_r}\right)^{\frac12}\\
&\quad\quad\quad\times
\max\left(\|e_{J_k}\|_r,(w_1\cdots w_{mJ_k+\rho})^{\frac 1m}\right)^{\frac{m-1}{2}}   \\
&\leq C \frac{\|e_\rho\|_r^{\frac 12}\cdot \|e_{J_k-3p+j+\rho}\|_r^\frac12}{(w_1\cdots w_{mJ_k-3p+j+\rho})^{\frac12}}\\
&\quad\quad\quad\times \max\left(\|e_{J_k}\|_r,(w_1\cdots w_{mJ_k+\rho})^{\frac 1m}\right)^{\frac{m-1}{2}},
\end{align*}
where the last line comes from the continuity of the product, more precisely from 
$$ \|e_{J_k-3p+l+\rho}\|_r \leq \|e_{J_k-3p+l}\|_r\cdot \|e_\rho\|_r.$$

Assume first that the maximum is attained for $\|e_{J_k}\|_r$. In that case, using (c) in a similar way we write
\[ \frac{\|e_{J_k-3p+j+\rho}\|_r}{ w_1\cdots w_{mJ_k-3p+j+\rho}} \|e_{J_k}\|_r^{m-1}\leq C \frac{\|e_{mJ_k-3p+j+\rho}\|_{q}}{w_1\cdots w_{mJ_k-3p+j+\rho}} \]
and the last parcel goes to zero. If the maximum is attained for $(w_1\cdots w_{mJ_k+\rho})^{\frac 1m}$, 
we now write
\begin{gather*}
\frac{\|e_{J_k-3p+j+\rho}\|_r (w_1\cdots w_{mJ_k+\rho})^{\frac{m-1}m}}{ w_1\cdots w_{mJ_k-3p+j+\rho}}
= \frac{\bigl(\|e_{J_k-3p+j+\rho}\|_r^m (w_1\cdots w_{mJ_k+\rho})^{m-1}\bigr)^{\frac{1}{m}}}{w_1\cdots w_{mJ_k-3p+j+\rho}}
\\
\leq C_1 \frac{\left ( (w_1\cdots w_{mJ_k+\rho})^{m-1} \|e_{mJ_k-m(3p-j)+m\rho} \|_{q}\right)^{\frac 1m} }{w_1\cdots w_{mJ_k-3p+j+\rho}}.
\end{gather*}
Now, using Corollary \ref{cor:wscauchy} for $n=mJ_k$, $N=3p$  and $s=3p-j$, we get the existence of $C_2>0$ and $q'\geq q$ (which does not depend on $k$) such that 
\begin{align*}
  (w_1\cdots w_{mJ_k+\rho} )^{m-1} \|e_{mJ_k-m(3p-j)+m\rho}\|_q &\leq C_2(w_1\cdots w_{mJ_k-3p+j+\rho})^{m-1}\\
  &\quad\quad\quad\times \|e_{mJ_k-3p+j+\rho}\|_{q'} 
\end{align*}
so that 
\[  \frac{\|e_{J_k-3p+j+\rho}\|_r (w_1\cdots w_{mJ_k+\rho})^{\frac{m-1}m}}{ w_1\cdots w_{mJ_k-3p+j\rho}}
\leq C_1 C_2^{\frac 1m} \left(\frac{\|e_{mJ_k-3p+j+\rho}\|_{q'}}{ w_1\cdots w_{mJ_k-3p+j+\rho}}\right)^{\frac 1m} \]
and this goes to zero.

Finally, let us prove \eqref{condconv:3}. The proof is very similar. Indeed, for all $l=0,\dots,p$, 
\begin{align*}
 &\left(\frac{\|e_{J_k-3p+l+\rho}\|_r}{w_1\cdots w_{mJ_k-3p+l+\rho}}\right)^{\frac m{2(m-1)}} (w_1\cdots w_{mJ_k+\rho})^{\frac 12}\\ 
&\quad\quad\quad\quad\leq \frac{\left(\|e_{J_k-3p+l+\rho}\|_r^m (w_1\cdots w_{mJ_k+\rho})^{m-1}\right)^{\frac1{2(m-1)}}}{(w_1\cdots w_{mJ_k-3p+l+\rho})^{\frac m{2(m-1)}}}\\
&\quad\quad\quad\quad\leq C_1 \frac{\left(\|e_{mJ_k-m(3p-l)+m\rho}\|_q (w_1\cdots w_{mJ_k+\rho})^{m-1}\right)^{\frac1{2(m-1)}}}
{(w_1\cdots w_{mJ_k-3p+l+\rho})^{\frac m{2(m-1)}}}\\
&\quad\quad\quad\quad\leq C_1 C_2^{\frac 1{2(m-1)}} \frac{\left( \|e_{mJ_k-3p+l+\rho}\|_{q'} (w_1\cdots w_{mJ_k-3p+l+\rho})^{m-1}\right)^{\frac 1{2(m-1)}}}
{(w_1\cdots w_{mJ_k-3p+l+\rho})^{\frac m{2(m-1)}}}
\\
&\quad\quad\quad\quad\leq C_1 C_2^{\frac 1{2(m-1)}}\left(\frac{\|e_{mJ_k-3p+l+\rho}\|_{q'}} {w_1\cdots w_{mJ_k-3p+l+\rho}}\right)^{\frac 1{2(m-1)}},
\end{align*}
which achieves the proof of \eqref{condconv:3}.

We now sketch briefly the proof when $m=1$. We still set  $\rho =\beta_2s_2+\dots +\beta_ds_d$ and we now consider a sequence $(J_k)$ satisfying 
$$ (w_1\cdots w_{J_k+l+\rho})^{-1}e_{J_k+l+\rho}\xrightarrow{k\to+\infty} 0$$
for all $l=0,\dots,p$.
We define
\begin{gather*}
u_1=x_1+\sum_{j=0}^pd_je_{J_k+j+\rho},\\
u_i=x_i+\eta_ie_{s_i},\text{  for  }i=2,...,d,
\end{gather*}
where 
$$
d_j=\frac{y_jw_1\dots w_j}{\eta_2^{s_2}\dots \eta_d^{s_d}w_1\dots w_{J_k+j+\rho}}.
$$
Setting $N=J_k+\rho$ for $k$ sufficiently large, it is easy to check that $u_i\in U_i$, $i=1,\dots ,d$, $B_w^N(u^{\alpha})=0$ if $\alpha \in A\setminus \{ \beta \}$, and $B_w^N(u^{\beta})\in V$, which concludes the proof. 
\end{proof}
85

\section{Frequently and upper frequently hypercyclic algebras}

\subsection{How to get upper frequently hypercyclic algebras}\label{sec:fhc}

In \cite{BoGre18}, following the proof made in \cite{BAYRUZSA} that the set of upper frequently hypercyclic vectors is either empty or residual,
the authors gave an analogue to Birkhoff's transitivity theorem for upper frequent hypercyclicity. We adapt it in order to get upper frequently hypercyclic algebras.

\begin{proposition}\label{prop:ufhccriterion}
Let $T$ be a continuous operator on an $F$-algebra $X$ satisfying the following condition: for each $I\in\mathcal P_f(\mathbb N)\backslash \{\varnothing\}$, there exists $m_0\in I$ such that, 
for each non-empty open subset $V$ of $X$ and each neighbourhood $W$ of the origin, there is $\delta >0$ such that for each non-empty open subset $U$ 
and each $N_0 \in \mathbb{N}$, there is $u\in U$ and $N\geq N_0$ satisfying
$$
\frac{1}{N+1}\card\left\{p\leq N: T^p(u^m)\in W \,\, \mbox{for} \,\, m\in I\backslash\{m_0\}\,\, \mbox{and} \,\, T^p(u^{m_0})\in V\right\}>\delta.
$$
Then $T$ admits an upper frequently hypercyclic algebra.
\end{proposition}

\begin{proof}
Let $(V_k)_{k}$ be a basis for the topology of $X$ and let $(W_j)_{j}$ be a basis of open neighbourhoods of the origin. 
By the assumption for each $I\in\mathcal P_f(\mathbb N)\backslash \{\varnothing\}$, there exists $m_0=m_0(I)$ such that, 
for each $k,j$, there is $\delta _{k,j,I}>0$ such that  for each non-empty open subset $U$ and each $N_0 \in \mathbb{N}$, there is $u\in U$ and $N\geq N_0$ satisfying
\begin{equation*}
\frac{1}{N+1}\card\left\{p\leq N: T^p(u^m)\in W_j \,\, \mbox{for} \,\, m\in I\backslash\{m_0\} \,\, \mbox{and} \,\, T^p(u^{m_0})\in V_k\right\}>\delta _{k,j,I}.
\end{equation*}
 We set 
\begin{align*}
A=\bigcap_{k,j,I,N_0\geq 1}\bigcup_{N\geq N_0}\Big\{u\in X&: \frac{1}{N+1}\card\big\{p\leq N: T^p(u^m)\in W_j \,\, \mbox{for} \,\, m\in I\backslash\{m_0\} \,\, \\ &\mbox{and} \,\, 
T^p(u^{m_0})\in V_k\big\}>\delta_{k,j,I}
\Big\}
\end{align*}
and show that $A$ is residual.

For fixed $k,j,I,N_0$ the set 
 \begin{align*}
 \bigcup_{N\geq N_0}\Bigg\{u\in X: &\ \frac{1}{N+1}\card\big\{p\leq N: T^p(u^m)\in W_j \,\, \mbox{for} \,\, m\in I\backslash\{m_0\} \,\, \mbox{and}\\
& \,\, T^p(u^{m_0})\in V_k\big\}>\delta_{k,j,I}\Big\}
 \end{align*}
is clearly open, and the definition of $\delta_{k,j,I}$ implies that it is also dense. By the Baire category theorem $A$ is residual.

Next, we check that if $u$ belongs to $A$ and if $P\in\CC[z]$ is not constant, with $P(0)=0$, then $P(u)\in UFHC(T)$.
We write $P(z)=\sum_{m\in I}\hat P(m)z^m$ for some $I\in\mathcal P_f(\mathbb N)\backslash \{\varnothing\}$ and $\hat P(m)\neq 0$ for all $m\in I$.
Since $UFHC(T)$ is invariant under multiplication by a scalar we may assume that $\hat P(m_0)=1$. 
Let $V$ be a non-empty open set and find $k,j$ such that $V_k+ (\card(I)-1)\| \hat P \|_{\infty}W_j\subset V$.
For each $N_0 \in \mathbb{N}$ there is $N\geq N_0$ such that 
$$
\frac{1}{N+1}\card\left \{p\leq N:  T^p(u^m)\in W_j \,\, \mbox{for} \,\, m\in I\backslash\{m_0\} \,\, \mbox{and} \,\, T^p(u^{m_0})\in V_k\right\}>\delta_{k,j,I}.
$$
But if $T^p(u^m) \in W_j$ for $m\in I\backslash\{m_0\}$ and $T^p(u^{m_0})\in V_k$,  then $T^p(P(u))\in V$. 
Therefore, 
\begin{align*}
\frac{1}{N+1}\card\left\{p\leq N: T^p(P(u))\in V\right\}> \delta_{k,j,I}
\end{align*}
which yields that $\overline{\textrm{dens}}(\{p\in \mathbb{N}: T^p(P(u))\in V\})> \delta_{k,j,I}>0$.
\end{proof}

We will apply this lemma either for $m_0(I)=\min(I)$ or $m_0=\max(I)$. The proposition gets the simpler forms:
\begin{corollary}\label{cor:ufhccriterion}
 Let $X$ be an $F$-algebra. If for each nonempty subset $V$ of $X$, for each neighbourhood $W$ of the origin, 
 for any positive integers $m_0<m_1$, there is $\delta >0$ such that for each nonempty open subset $U$ 
and each $N_0 \in \mathbb{N}$, there is $u\in U$ and $N\geq N_0$ satisfying
$$
\frac{1}{N+1}\card\left\{p\leq N: T^p(u^m)\in W \,\, \mbox{for} \,\, m\in \{m_0+1,\dots,m_1\}\,\, \mbox{and} \,\, T^p(u^{m_0})\in V\right\}>\delta
$$
then $T$ admits an upper frequently hypercyclic algebra.
\end{corollary}

\begin{corollary}\label{cor:ufhccriterionmax}
  Let $X$ be an $F$-algebra. If for each nonempty subset $V$ of $X$, for each neighbourhood $W$ of the origin, 
 for any positive integer $m$, there is $\delta >0$ such that for each nonempty open subset $U$ 
and each $N_0 \in \mathbb{N}$, there is $u\in U$ and $N\geq N_0$ satisfying
$$
\frac{1}{N+1}\card\left\{p\leq N: T^p(u^n)\in W \,\, \mbox{for} \,\, n\in \{1,\dots,m-1\}\,\, \mbox{and} \,\, T^p(u^{m})\in V\right\}>\delta
$$
then $T$ admits an upper frequently hypercyclic algebra.
\end{corollary}

\subsection[Upper frequently hypercyclic algebras I]{Existence of upper frequently hypercyclic algebras for weighted backward shifts - coordinatewise products}

We intend to  apply the previous method to backward shift operators and prove Theorem \ref{thm:ufhcws}. The unconditional convergence of the series $\sum_{n\geq 1}(w_1\cdots w_n)^{-1}e_n$ will be used throughout the following lemma.
\begin{lemma}\label{lem:unconditionalfhc}
Let $(w_n)$ be a weight sequence such that $\sum_{n\geq 1}(w_1\cdots w_n)^{-1}e_n$ converges unconditionally. Then, for all $\veps>0$, for all $p>0$, for all $M>0$, there exists $N\geq p$ such that, for each sequence of complex numbers $(y(n,l))_{n\geq N,\ 0\leq l\leq p}$ with $|y(n,l)|\leq M$ for all $n,l$, then 
\[ \left\|\sum_{n\geq N}\sum_{l=0}^p \frac{y(n,l)}{w_{l+1}\cdots w_{n+l}}e_{n+l}\right\|\leq\veps. \]
\end{lemma}
\begin{proof}
We first observe that the convergence of the series involved follows from the unconditional convergence of each series $\sum_{n\geq 1}(w_{l+1}\cdots w_{n+l})^{-1}e_{n+l}$. 
Setting $z(n,l)=y(n,l)/M$ and using the triangle inequality, the (almost) homogeneity of the $F-$norm implies that
\[\left\|\sum_{n\geq N}\sum_{l=0}^p \frac{y(n,l)}{w_{l+1}\cdots w_{n+l}}e_{n+l}\right\|\leq (M+1)\sum_{l=0}^p \left\|\sum_{n\geq N}\frac{z(n,l)}{w_{l+1}\cdots w_{n+l}}e_{n+l}\right\|.\]
The existence of an $N$ such that the last term is less than $\veps$ now follows directly from the unconditional convergence of the series $\sum_{n}(w_{l+1}\cdots w_{n+l})^{-1}e_{n+l}$
(see the introduction).
\end{proof}

\begin{proof}[Proof of Theorem \ref{thm:ufhcws}]
Let $m_0<m_1$ be two positive integers.  Let $V,W\subset X$ be open and non-empty with $0\in W$.
Let $p\geq 0$, $\veps>0$ and $v=\sum_{l=0}^p v_l e_l$ be such that $B(v,\veps)\subset V$ and $B(0,2\veps)\subset W$. We also set $M=\max(1,\|v\|_\infty)^{m_1/m_0}$. Let $N\in\mathbb N$ be given by Lemma \ref{lem:unconditionalfhc} for these values of $\veps$, $p$ and $M$. Without loss of generality, we may assume $N>p$. We set $\delta=\frac1{2N}$. Let $U\subset X$ be open and non-empty and let $x\in U$ with finite support. We also define, for $k\geq 0$, 
$$v(k)=\sum_{l=0}^p \frac{v_l^{1/m_0}}{(w_{l+1}\cdots w_{k+l})^{1/m_0}}e_{k+l}.$$
Let $N_1$ be very large (precise conditions on it will be given later; 
for the moment we just assume that $N_1$ is bigger than the maximum of the support of $x$). We finally set
$$u=x+\sum_{k\geq N_1} v(Nk).$$
The unconditional convergence of the series $\sum_k (w_{l+1}\cdots w_{k+l})^{-1/m_0}e_{k+l}$ ensures that $u$ is well-defined and that $\|u-x\|<\veps$ provided $N_1$ is large enough. 
Let now $m\in\{m_0,\cdots,m_1\}$ and $j\geq N_1$. Then, since $x$ and the $v(kN)$, $k\geq N_1$, have pairwise disjoint support,
\begin{eqnarray*}
B_w^{Nj}u^m&=&\sum_{l=0}^p \frac{v_l^{m/m_0}}{(w_{l+1}\cdots w_{jN+l})^{\frac{m}{m_0}-1}}e_l\\
&&\quad\quad+\sum_{k>j}\sum_{l=0}^p \frac{v_l^{m/m_0}w_{(k-j)N+l+1}\cdots w_{kN+l}}{(w_{l+1}\cdots w_{kN+l})^{\frac m{m_0}}}e_{(k-j)N+l}\\
&=:&z(1,j,m)+z(2,j,m).
\end{eqnarray*}	
If $m=m_0$, then $z(1,j,m)=v$ whereas, if $m\in\{m_0+1,\dots,m_1\}$, then since $|v_l|^{m/m_0}\leq M$ 
and since the sequences $(w_{l+1}\cdots w_{l+n})_n$ go to $+\infty$ (recall that $X$ has a continuous norm),
we may adjust $N_1$ so that $\|z(1,j,m)\|<\veps$ for all $j\geq N_1$. On the other hand, we may write
\[ z(2,j,m)=\sum_{n\geq N}\sum_{l=0}^p \frac{y(j,n,l,m)}{w_{l+1}\cdots w_{n+l}}e_{n+l}\]
where, for $s=k-j\geq 1$, $l=0,\dots,p$, 
$$y(j,sN,l,m)=\frac{v_l^{m/m_0}}{(w_{l+1}\cdots w_{(s+j)N})^{\frac m{m_0}-1}}$$
and $y(j,n,l,m)=0$ if $n$ is not a multiple of $N$. 
Again taking $N_1$ large enough guarantees that $|y(j,n,l,m)|\leq M$ for all $j\geq N_1$, all $n\geq N$, all $l=0,\dots,p$ and all $m=m_0,\dots,m_1$.
By the choice of $N$, we get $\|z(2,j,m)\|<\veps$ for all $j\geq N_1$. Summarizing, we have proved that, for all $j\geq N_1$, $B_w^{Nj}u^{m_0}\in V$ and $B_w^{Nj}u^m\in W$
for all $m\in\{m_0+1,\dots,m_1\}$. Hence we may apply Corollary \ref{cor:ufhccriterion} to prove that $B_w$ supports an upper frequently hypercyclic algebra.	
\end{proof}

The Theorem \ref{thm:ufhcws} can be applied to the following examples: $\lambda B$ on $\ell_p$ for $\lambda>1$, 
$\lambda D$ on $H(\CC)$ for $\lambda>0$ or $B_w$ on $c_0$ with $w_n=1+\lambda/n$, $\lambda>0$. Regarding this last weight, on $\ell_p$, $B_w$ is upper
frequently hypercyclic if and only if $\lambda>1/p$. However we do not know the answer to the following question because of the divergence
of $\sum_n (w_1\cdots w_n)^{-1/m}e_n$ for $m>\lambda p$.
\begin{question}
Let $X=\ell_p$ and $w_n=1+\lambda/n$ for $\lambda>1/p$. Does $B_w$ supports an upper frequently hypercyclic algebra? 
\end{question}

On $\ell_p$, it is known that $B_w$ is (upper) frequently hypercyclic if and only if $\sum_n (w_1\cdots w_n)^{-p}$ is convergent (see \cite{BAYRUZSA}). 
\begin{question}
Let $X=\ell_p$ endowed with the coordinatewise product and let $w=(w_n)$ be an admissible weight sequence. Assume that $B_w$ supports an upper frequently hypercyclic algebra. Does this imply that $\sum_n (w_1\cdots w_n)^{-\gamma}$ is convergent for all $\gamma>0$?
\end{question}

\subsection[Upper frequently hypercyclic algebras II]{Existence of upper frequently hypercyclic algebras for weighted backward shifts - convolution product}

We now study the existence of an upper frequently hypercyclic algebra for weighted backward shifts when the underlying Fréchet algebra is endowed
with the convolution product. We shall give a general statement encompassing the case of the multiples of the backward shift and of the derivation operator.

\begin{theorem}\label{thm:ufhcconvolution}
 Let $X$ be a regular Fréchet sequence algebra for the Cauchy product and let $(w_n)$ be an admissible weight sequence. Assume that
 \begin{enumerate}[(a)]
  \item $\sum_{n\geq 1}(w_1\cdots w_n)^{-1}e_n$ converges unconditionally.
  \item for all $m\geq 2$, there exists $c\in (0,1)$ such that 
  \[ \lim_{\sigma\to+\infty} \sup_{z\in c_{00}\cap B_{\ell_\infty}}
   \left\|\sum_{n\geq c\sigma} \frac{z_n (w_1\cdots w_{m\sigma})^{(m-1)/m}}{w_1\cdots w_{(m-1)\sigma+n}}e_n\right\|=0. 
  \]
Then $B_w$ admits an upper frequently hypercyclic algebra.
 \end{enumerate}
\end{theorem}

\begin{proof}
 We shall prove that the assumptions of Corollary \ref{cor:ufhccriterionmax} are satisfied. For $m=1$, this follows from condition (a) which implies that $T$ admits a dense
 set of (upper) frequently hypercyclic vectors. Thus, let us assume that $m\geq 2$ and let $c\in (0,1)$ be given by (b). We also consider $d\in(c,(1+c)/2)\subset (0,1)$. Let $V$ be a non-empty
 open subset of $X$, $W$ a neighbourhood of $0$, $y=\sum_{l=0}^p y_l e_l\in V$ and $\eta>0$ such that $B(0,2\eta)+y\subset V$. Let finally $q>p$ be such that, for all $z\in\ell_\infty$ with $\|z\|_\infty\leq \|y\|_\infty\big(\max(1,w_1,\dots,w_p)\big)^{p+1}$,
 \[ \left\|\sum_{n\geq q}\frac{z_n}{w_1\cdots w_n}e_n\right\|<\eta. \]
 We intend to prove that, for each non-empty open subset $U$ and each $N_0\in\NN$,
 there is $u\in U$ and $N\geq N_0$ satisfying
 \[ \frac1{N+1}\card\left\{s\leq N:B_w^s(u^n)\in W\textrm{ for }n<m\textrm{ and }B_w^s(u^m)\in V\right\}\geq\frac{d-c}{2\big((m-1)q+qd\big)}. \]
 
 More precisely, we shall prove that, for all $\sigma$ large enough, setting 
 \[ E_\sigma=\left\{(m-1)q\sigma+qj: c\sigma\leq j<d\sigma\right\} ,\]
 there exists $u\in U$ such that, for all $s\in E_\sigma$, $B_w^s(u^n)=0$ for $n<m$ and $B_w^s(u^m)\in V$. Since
 \[ \lim_{\sigma\to+\infty} \frac{\card(E_\sigma)}{\max(E_\sigma)}=\frac{d-c}{(m-1)q+qd} ,\]
 we will get the claimed result.
 
 We thus fix $x\in U$ with finite support (we denote by $p'$ the maximum of the support of $x$) and let $\sigma>0$ be such that $p'<c\sigma$. Inspired
 by the proof of Theorem \ref{thm:wscauchy}, we set
 \[ u=x+\sum_{j=c\sigma}^{d\sigma-1} \sum_{l=0}^p d_{j,l}e_{qj+l}+\veps e_{q\sigma} \]
 where
 \begin{align*}
  \veps &= \frac{1}{(w_1\cdots w_{mq\sigma})^{1/m}}\\
  d_{j,l}& = \frac{y_l}{m\veps^{m-1} w_{l+1}\cdots w_{(m-1)q\sigma+qj+l}}.
 \end{align*}
 Let us first prove that, provided that $\sigma$ is large enough, $u$ belongs to $U$.
 Let $r\geq 1$. Since $X$ is regular, there exists $\rho\geq r$ and $C>0$ such that 
 \[ \left\| \veps e_{q\sigma}\right\|_r\leq \frac{C}{(w_1\cdots w_{mq\sigma})^{1/m}}\left\|e_{mq\sigma}\right\|^{1/m}_\rho\xrightarrow{\sigma\to+\infty}0. \]
Furthermore,
\begin{align*}
 \left\| \sum_{j=c\sigma}^{d\sigma-1} \sum_{l=0}^p d_{j,l}e_{qj+l} \right\| &=
 \left\|\sum_{j=c\sigma}^{d\sigma-1}\sum_{l=0}^p \frac{y_l w_1\cdots w_l (w_1\cdots w_{mq\sigma})^{(m-1)/m}}{m w_1\cdots w_{(m-1)q\sigma+qj+l}}e_{qj+l}\right \|\\
 &= \left\|\sum_{n\geq cq\sigma}\frac{z_n (w_1\cdots w_{mq\sigma})^{(m-1)/m}}{w_1\cdots w_{(m-1)q\sigma+n}}e_n\right\|
\end{align*}
for some eventually null sequence $(z_n)$ such that 
$$\|z\|_\infty\leq \|y\|_\infty \big(\max(1,w_1,\dots,w_p)\big)^{p}/m.$$
Assumption (b) allows us to conclude that $\sum_{j=c\sigma}^{d\sigma-1} \sum_{l=0}^p d_{j,l}e_{qj+l}$ tends to zero as $\sigma$ goes to $+\infty$.

Observe now that, for $n<m$, the support of $u^n$ is contained in $[0,nq\sigma]$ so that, for $s\in E_\sigma$, $B_w^s(u^n)=0$. On the other hand,
\[ u^m = z+m\veps^{m-1} \sum_{j=c\sigma}^{d\sigma-1} \sum_{l=0}^p d_{j,l}e_{(m-1)q\sigma+qj+l}+\veps^m e_{mq\sigma} \]
with $\supp(z)\subset [0,(m-2)q\sigma+2qd\sigma]\cup [0,(m-1)q\sigma+p']$. It is not difficult to see that, because $d<(1+c)/2$, 
$\max(\supp(z))\leq (m-1)q\sigma+qc\sigma$ for $\sigma$ large enough. Thus, for any $s=(m-1)q\sigma+qk\in E_\sigma$, 
\[ B_w^s(u^m)=y+\sum_{j=k+1}^{d\sigma-1}\sum_{l=0}^p \frac{y_l}{w_{l+1}\cdots w_{q(j-k)+l}}e_{q(j-k)+l}+\frac{1}{w_1\cdots w_{q\sigma-qk}}e_{q\sigma-qk}. \]
We handle the second term of the right hand side of the equality by writing
\[ \sum_{j=k+1}^{d\sigma-1} \sum_{l=0}^p \frac{y_l}{w_{l+1}\cdots w_{q(j-k)+l}}e_{q(j-k)+l}
=\sum_{n\geq q} \frac{z_n}{w_1\cdots w_n}e_n
\]
for some sequence $(z_n)$ such that $\|z\|_\infty\leq \|y\|_\infty \big(\max(1,w_1,\dots,w_p)\big)^p$. By our choice of $q$, this has $F$-norm less than $\eta$. Finally, $\|e_{q\sigma-qk}\|/(w_1\cdots w_{q\sigma-qk})$ becomes also less than $\eta$ provided $\sigma$, and
thus $q\sigma-qd\sigma$, becomes large enough.
\end{proof}

Assumption (a) in Theorem \ref{thm:ufhcconvolution} is what we need to get an (upper) frequently hypercyclic vector. Assumption (b) is also an assumption around unconditional convergence of the series $\sum_{n\geq 1}(w_1\cdots w_n)^{-1}e_n$. This looks clearer by writing 
\[ \sum_{n\geq c\sigma} \frac{z_n(w_1\cdots w_{m\sigma})^{(m-1)/m}}{w_1\cdots w_{n+(m-1)\sigma}}e_n=\sum_{n\geq c\sigma} \frac{z_n(w_1\cdots w_{m\sigma})^{(m-1)/m}}{w_{n+1}\cdots w_{n+(m-1)\sigma}}\times\frac{e_n}{w_1\cdots w_n}. \]
Although it looks quite technical, it is satisfied by three natural examples (where we always endow the $F$-algebra with the Cauchy product).
\begin{example}Let $X=\ell_1$ and $w_n=\lambda>1$ for all $n\in\NN$. Then $B_w$ supports an upper frequently hypercyclic algebra.
\end{example}
\begin{proof}
 The situation is very simple here because, for all $n\geq 1$, 
 \[  \frac{(w_1\cdots w_{m\sigma})^{(m-1)/m}} {w_{n+1}\cdots w_{n+(m-1)\sigma}}=1, \]
 so that (b) is clearly satisfied. 
\end{proof}
\begin{example}
 Let $X=\ell_1$ and $w_1\cdots w_n=\exp(n^\alpha)$ for all $n\in\NN$ and some $\alpha\in(0,1)$. Then $B_w$ supports an upper frequently hypercyclic algebra.
\end{example}
\begin{proof}
 That (b) is satisfied follows from the classical asymptotic behavior
 \begin{equation}\label{eq:wsexpnalpha}
 \sum_{n\geq N}\exp\left(-n^\alpha\right)\sim_{N\to+\infty}\frac 1\alpha N^{1-\alpha}\exp(-N^\alpha).
 \end{equation}
 Assuming \eqref{eq:wsexpnalpha} is true, we just write
\begin{align*}
& \left\|\sum_{n\geq c\sigma} \frac{z_n (w_1\cdots w_{m\sigma})^{(m-1)/m}}{w_1\cdots w_{(m-1)\sigma+n}}e_n\right\| \\
&\quad\quad  = \exp\left(\frac{m-1}m  m^\alpha\sigma^\alpha\right)\sum_{n\geq c\sigma}\exp\left( -\left((m-1)\sigma+n\right)^\alpha\right)\\
 &\quad\quad \sim_{+\infty}C\sigma^{1-\alpha} \exp\left( \left(\frac {m-1}m m^\alpha-(m-1+c)^\alpha\right)\sigma^\alpha\right).
\end{align*}
Assumption (b) is satisfied for $c$ close enough to $1$,  since in that case
 \[ (m-1+c)^\alpha > \frac{(m-1)m^{\alpha}}{m}. \]
 For the sake of completeness, we just mention that \eqref{eq:wsexpnalpha} follows from the formula of integration by parts:
 \[ \int_{N}^{+\infty}\exp(-x^\alpha)dx=\frac 1\alpha N^{1-\alpha}\exp(-N^\alpha)+\frac{1-\alpha}{\alpha}\int_{N}^{+\infty}x^{-\alpha}\exp(-x^\alpha)dx. \]
 \end{proof}
 \begin{example}The derivation operator $D$ supports an upper frequently hypercyclic algebra on $H(\CC)$.
 \end{example}
 \begin{proof}
It is sufficient to verify (b) replacing $\|\cdot\|$ by any seminorm $\|\cdot\|_r$. Now, 
 \begin{align*}
 \left\|\sum_{n\geq c\sigma} \frac{z_n (w_1\cdots w_{m\sigma})^{(m-1)/m}}{w_1\cdots w_{(m-1)\sigma
 +n}}e_n\right\|_r & =
 \sum_{n\geq c\sigma}\frac{(m\sigma)!^{(m-1)/m}}{((m-1)\sigma+n)!}r^n\\
 &= \frac{(m\sigma)!^{(m-1)/m}}{r^{(m-1)\sigma}}\sum_{n\geq c\sigma}\frac{r^{(m-1)\sigma+n}}{((m-1)\sigma+n)!}\\
 &\leq C r^{c\sigma}\frac{(m\sigma)!^{(m-1)/m}}{(m-1+c)\sigma}.
 \end{align*}
Since for all $\veps>0$ Stirling's formula implies
\begin{gather*}
(m\sigma)!^{(m-1)/m}\leq C_\veps \sigma^{(m-1+\veps)\sigma}\\
((m-1+c)\sigma)!\geq C_\veps\sigma^{(m-1+c-\veps)\sigma},
\end{gather*}
choosing $\veps<2c$ it follows that, for all $c\in(0,1)$ and all $r\geq1$, we have
\begin{align*}
\left\|\sum_{n\geq c\sigma} \frac{z_n (w_1\cdots w_{m\sigma})^{(m-1)/m}}{w_1\cdots w_{(m-1)\sigma
 +n}}e_n\right\|_r & \leq C\frac{r^{c\sigma}}{\sigma^{(c-2\veps)\sigma}}\xrightarrow{\sigma\to+\infty}0.
\end{align*} 
Hence, assumption (b) is verified.
\end{proof}
Another natural operator that could admit an upper frequently hypercyclic algebra is the backward shift $B_w$ with $w_n=\left(1+\frac \lambda n\right)$, $\lambda>1$, 
acting over $\ell_1$ with the convolution product. Unfortunately, for this weight, the assumptions in Theorem \ref{thm:ufhcconvolution} are not verified.
\begin{question}
Let $X=\ell_1$ endowed with the convolution product and $w_n=\left(1+\frac \lambda n\right)$, $\lambda>1$. Does $B_w$ admit an upper frequently hypercyclic algebra?
\end{question}

We can ask a similar question for convolution operators $\phi(D)$ on $H(\CC)$, $|\phi(0)|<1$, which are frequently hypercyclic and admit a hypercyclic algebra.
\begin{question}
Let $X=H(\CC)$ and let $\phi:\CC\to\CC$ be a nonconstant entire function with exponential type, not a multiple of an exponential function, with $|\phi(0)|<1$. Does $\phi(D)$ supports an upper frequently hypercyclic algebra?
\end{question}

\subsection{Weighted shifts with a frequently hypercyclic algebra on $\omega$}

Despite the result of Falc\'o and Grosse-Erdmann, it is not so difficult to exhibit operators supporting a frequently hypercyclic algebra
if we work on the big space $\omega$.

\begin{theorem}\label{thm:fhcalgomega}
Let $w=(w_n)_{n\geq 1}$ be a weight sequence such that  $(w_{1}\cdots w_{n})$ either tends to $+\infty$ or to 0.
Then $B_w$, acting on $\omega$ endowed with the coordinatewise product, supports a frequently hypercyclic algebra. 
\end{theorem}
\begin{proof}
We first assume that $(w_1\cdots w_n)$ tends to $+\infty$ and observe that this clearly implies that, for all $l\geq 0$, $(w_{l+1}\cdots w_{n+l})$ tends to $+\infty$. 
Let $(v(p),m(p))$ be a dense sequence in $\omega\times\NN$, where each $v(p)$ has finite support contained in $[0,p]$. We then write $v(p)=\sum_{l=0}^p v_l(p)e_l$. 
For $(n,p)\in\NN^2$, we define
\[ y(n,p)=\sum_{l=0}^p \frac{v_l(p)^{1/m(p)}}{(w_{l+1}\cdots w_{n+l})^{1/m(p)}}e_{n+l} .\]
By \cite[Lemma 6.19]{BM09} (see the forthcoming Lemma \ref{lem:setsfhc}), there exists a sequence $(A(p))$ of pairwise disjoint subsets of $\NN$, with positive lower density, and such that $|n-n'|\geq p+q+1$ whenever $n\neq n'$ and $(n,n')\in A(p)\times A(q)$. In particular, the vectors $y(n,p)$ for $p\in\NN$ and $n\in A(p)$ have disjoint support. Hence, we may define $u=\sum_{p\in\NN}\sum_{n\in A(p)}y(n,p)$ and we claim that $u$ generates a frequently hypercyclic algebra. 

Indeed, let $P\in\mathbb C[z]$ be non-constant with $P(0)=0$, $P(z)=\sum_{m=m_0}^{m_1}\hat P(m) z^m$, $\hat P(m_0)\neq 0$,
and let $V$ be a non-empty open subset of $\omega$. Let $p\in\NN$, $\veps>0$ be such that $m(p)=m_0$ and any vector $x\in\omega$ satisfying $|x_l-\hat P(m_0)v_l(p)|<\veps$ for all $l=0,\dots,p$ belongs to $V$. Now, for $l=0,\dots,p$ and $n\in A(p)$, 
\[ \big(B_w^n P(u)\big)_l=\hat P(m_0) v_l(p)+\sum_{m=m_0+1}^{m_1}\frac{\hat P(m) v_l(p)^{\frac m{m_0}}}{(w_{l+1}\cdots w_{n+l})^{\frac m{m_0}-1}}. \]
Since $(w_{l+1}\cdots w_{n+l})$ tends to $+\infty$ for all $l$, $B_w^n P(u)$ belongs to $V$ for all $n$ in a cofinite subset of $A(p)$. Hence, $P(u)$ is a frequently hypercyclic vector for $B_w$.

The proof is completely similar if we assume that $(w_1\cdots w_n)$ tends to $0$. The only difference is that the dominant term is now given by the term of highest degree of $P$, namely we choose $p$
such that $m(p)=m_1$ and we write
\[ \big(B_w^n P(u)\big)_l=\hat P(m_1) v_l(p)+\sum_{m=m_0}^{m_1-1}\frac{\hat P(m) v_l(p)^{\frac m{m_1}}}{(w_{l+1}\cdots w_{n+l})^{\frac m{m_1}-1}}. \]
We will conclude because, for $m<m_1$, $(w_{l+1}\cdots w_{n+l})^{\frac m{m_1}-1}$ tends to $+\infty$.
Details are left to the reader.
\end{proof}

The unweighted backward shift on $\omega$ (still endowed with the coordinatewise product) supports a frequently hypercyclic algebra. Indeed, more generally, let $T$ be a multiplicative operator on an $F$-algebra $X$ with the property that
for every non-zero polynomial $P$ vanishing at the origin, the map
$$
\tilde{P}:X\rightarrow X, x\mapsto P(x)
$$
has dense range. Then if $T$ is frequently hypercyclic, it supports a frequently hypercyclic algebra. The reason for that is the simple observation that if $U$ is a non-empty open set of
$X$ and $P$ a non-zero polynomial vanishing at the origin,
\[ \left\{n\in\NN:\ T^n(P(x))\in U\right\}=\left\{n\in\NN:\ T^n x\in \tilde{P}^{-1}(U)\right\}. \]

From the same observation we may conclude that the translation operators $T_a$ acting on $C^{\infty}(\mathbb{R},\mathbb{C})$ for $a\in \mathbb{R}, a\neq 0$, admit a frequently hypercyclic algebra. The fact that $C^{\infty}(\mathbb{R},\mathbb{C})$ has the above mentioned property is proven in \cite[Proposition 20]{BCP18}.

\subsection[Sequence of sets with positive lower density]{A sequence of sets with positive lower density which are very far away from each other}

The remaining part of this section is devoted to the proof of Theorem \ref{thm:afhcc0}.
The starting point to exhibit frequently hypercyclic vectors is the following lemma on the existence of subsets of $\NN$ with positive lower density which are sufficiently separated.

\begin{lemma}[Lemma 6.19 in \cite{BM09}]\label{lem:setsfhc} Let $(a(p))$ be any sequence of positive real numbers. Then one can find a sequence $(A(p))$ of pairwise disjoint subsets of $\NN$ such that
\begin{enumerate}[(i)]
\item each set $A(p)$ has positive lower density;
\item $\min A(p)\geq a(p)$ and $|n-n'|\geq a(p)+a(q)$ whenever $n\neq n'$ and $(n,n')\in A(p)\times A(q)$.
\end{enumerate}
\end{lemma}

To produce a frequently hypercyclic algebra for a weighted shift on $c_0$, we will need a refined version of this lemma where we add new conditions of separation.

\begin{theorem}\label{thm:setsafhc}
 Let $(a(p))$ be any sequence of positive real numbers. Then one can find a sequence $(A(p))$ of pairwise disjoint subsets of $\NN$ such that
\begin{enumerate}[(i)]
\item each set $A(p)$ has positive lower density;
\item $\min A(p)\geq a(p)$ and $|n-n'|\geq a(p)+a(q)$ whenever $n\neq n'$ and $(n,n')\in A(p)\times A(q)$.
\item for all $C>0$, there exists $\kappa>0$ such that, for all $(n,n')\in A(p)\times A(q)$ with $p\neq q$ and $\max(n,n')\geq \kappa$, then $|n-n'|\geq C$.
\end{enumerate}
\end{theorem}

The proof of this theorem is rather long. The strategy is to construct a sequence of sets satisfying only (i) and (iii), and then to modify them to add (ii).
We begin with two sets.
\begin{lemma}\label{lem:setsafhc1}
Let $E\subset\NN$ be a set with positive lower density. There exist $A,B\subset E$ disjoint, with positive lower density, and such that, for all $C>0$, there exists $\kappa>0$ such that, for all $n\in A$ and all $n'\in B$ with $\max(n,n')\geq \kappa$, then $|n-n'|\geq C$.
\end{lemma}

\begin{proof}
We write $E=\{n_j:\ j\in\NN\}$ in an increasing order. We set, for $k\geq 1$, $u_k=k$, $v_k=\lfloor \sqrt k\rfloor$. We define sequences $(M_k)$, $(N_k)$, $(P_k)$ and $(Q_k)$ by setting $M_1=1$ and, for $k\geq 1$, 
\[ N_k=M_k+u_k,\ P_k=N_k+v_k,\ Q_k=P_k+u_k,\ M_{k+1}=Q_k+v_k.\]
We then define
\[I=\bigcup_k [M_k,N_k),\ \ J=\bigcup_k [P_k,Q_k), \]
\begin{align*}
A=\left\{n_j:\ j\in I\right\},\ \ 
B=\left\{n_j:\ j\in J\right\}.
\end{align*}
The sets $I$ and $J$ have positive lower density. Indeed, for $N\in\NN$, let $k$ be such that $N\in [M_k,M_{k+1})$. Then 
\[ \frac{\card(I\cap [1,\dots,N])}N\geq \frac{u_1+\cdots+u_{k-1}}{2(u_1+\cdots+u_k+v_1+\cdots+v_k)}\geq\frac 14 \]
provided $k$ is large enough. The same is true for $J$. Since $E$ has positive lower density, this yields that $A$ and $B$ have positive lower density. Moreover, let $C>0$. 
There exists $k\geq 0$ such that $v_{k-1}\geq C$. We set $\kappa=n_{M_k}$. Let $(n,n')\in A\times B$ with $\max(n,n')\geq \kappa$. 
Assume for instance that $n\geq n_{M_k}$ and write $n=n_j$, $n'=n_{j'}$. Then $j\geq M_k$ and the construction of the sets $I$ and $J$ ensure that $j'$ does not belong to $[j-v_{k-1},j+v_{k-1}]$.
Thus, $|n-n'|\geq |j-j'|\geq C$.
\end{proof}

It is not difficult to require that (ii) in Theorem \ref{thm:setsafhc} holds when we restrict ourselves to $p=q$.
\begin{lemma}\label{lem:setsafhc3}
Let $A\subset\NN$ with positive lower density and $a>0$. There exists $B\subset A$ with positive lower density, $\min(B)\geq a$ and $|n-n'|\geq a$ for all $n,n'\in B$, $n\neq n'$.
\end{lemma}
\begin{proof}
Write $A=\{n_j:\ j\in\NN\}$ in an increasing order and define $B=\{n_{ka}:\ k\in\NN\}$.
\end{proof}

We then go inductively from two sets to a sequence of sets.

\begin{lemma}\label{lem:setsafhc2}
There exists a sequence $(A(p))$ of pairwise disjoint subsets of $\NN$ such that
\begin{enumerate}[(i)]
\item each set $A(p)$ has positive lower density;
\item for all $C>0$, there exists $\kappa>0$ such that, for all $(n,n')\in A(p)\times A(q)$ with $p\neq q$ and $\max(n,n')\geq \kappa$, then $|n-n'|\geq C$.
\end{enumerate}
\end{lemma}
\begin{proof}
We shall construct by induction two sequences of sets $(A(p))$ and $(B(p))$ and a sequence of integers $(\kappa_k)$ such that, at each step $r$, 
\begin{enumerate}[(a)]
\item for all $1\leq p\leq r$, $A(p)$ and $B(p)$ are disjoint and have positive lower density.
\item for all $1\leq p<q\leq r$, $A(q)\subset B(p)$ and $B(q)\subset B(p)$.
\item for all $C>0$, there exists $\kappa>0$ such that, for all $1\leq p\leq q\leq r$, for all $n\in A(p)$ and $n'\in B(q)$, $\max(n,n')\geq\kappa\implies |n-n'|\geq C$. 
\item for all $k\in\{1,\dots,r\}$, for all $1\leq p\leq q\leq r$, for all $n\in A(p)$ and $n'\in B(q)$, $\max(n,n')\geq \kappa_k\implies |n'-n|\geq k$.
\end{enumerate}
It is straightforward to check that the resulting sequence $(A(p))$ satisfies the conclusions of Lemma \ref{lem:setsafhc2}. 
Observe nevertheless that it is condition (d) together with the inclusion $A(q)\subset B(p)$ for $q>p$ which gives (ii) in this lemma (which is uniform with respect to $p$ and $q$). Condition (c) is only helpful for the induction hypothesis.

We initialize the construction by applying Lemma \ref{lem:setsafhc1} to $E=\NN$. We set $A(1)=A$ and $B(1)=B$ which satisfy (a), (b) and (c).
In particular, applying (c) for $C=1$ we find some $\kappa$ that we call $\kappa_1$. 

Assume now that the construction has been done until step $r$ and let us perform it for step $r+1$. Let $E$ be a subset of $B(r)$ with positive lower density and $|n-n'|\geq r+1$ 
provided $n\neq n'$ are in $E$. We apply Lemma \ref{lem:setsafhc1} to this set $E$ and we set $A(r+1)=A$ and $B(r+1)=B$, so that (a) and (b) are clearly satisfied.
Upon taking a maximum, (c) is also easily satisfied:
indeed, the only case which is not settled by the induction hypothesis is $p=q=r+1$ (when $p<r+1$ and $q=r+1$, use $B(q)\subset B(r)$);
this case is solved by the construction of $A(r+1)$ and $B(r+1)$. 

The proof of (d) is slightly more delicate. For $k=1,\dots,r$, we have to verify that for $1\leq p\leq r+1$, $n\in A(p)$ and $n'\in B(r+1)$, $\max(n,n')\geq \kappa_k\implies |n-n'|\geq k$. 
When $p\leq r$, again this follows from $B(r+1)\subset B(r)$. For $p=r+1$, this follows from $A(r+1),B(r+1)\subset E$ 
and the fact that distinct elements of $E$ have distance greater than or equal to $r+1$. Finally, applying (c) for $C=k+1$, we define $\kappa_{k+1}$.
\end{proof}

We need now to ensure property (ii) in Theorem \ref{thm:setsafhc}. This will be done again inductively, the main step being the following lemma.
\begin{lemma}\label{lem:setsafhc4}
Let $(A(p))$ be a sequence of pairwise disjoint subsets of $\NN$ with positive lower density and let $(a(p))$ be a sequence of positive real numbers. There exists a sequence $(B(p))$ of subsets of $\NN$ with positive lower density such that each $B(p)$ is contained in $A(p)$ and, for all $n\in B(1)$, for all $n'\in B(p)$, $p\geq 2$, $|n-n'|\geq a(1)+a(p)$.
\end{lemma}
\begin{proof}
Since $A(1)$ has positive lower density, there exists $N\in\NN$ and $\delta>0$ such that, for all $n\geq N$, 
$$\frac{\card\big(A(1)\cap[0,\dots,n]\big)}{n+1}\geq\delta.$$
For $p\geq 2$, let $B(p)$ be a subset of $A(p)$ such that, for all $n\in\NN$, 
$$\frac{\card\big((B(p)+[-a(1)-a(p),a(1)+a(p)])\cap [0,\dots,n]\big)}{n+1}<\frac{\delta}{2^p}$$
but $B(p)$ still has positive lower density. This is possible if, writing $A(p)=\{n_j;\ j\in\NN\}$, we set $B(p)=\{n_{ka}; k\geq 1\}$ for some sufficiently large $a$. We then define $B(1)=A(1)\backslash \bigcup_{p\geq 2}\big(B(p)+[-a(1)-a(p),a(1)+a(p)]\big)$. Then, for all $n\in B(1)$ and all $n'\in B(p)$, $p\geq 2$, one clearly has $|n-n'|\geq a(1)+a(p)$ whereas, for all $n\geq N$, 
\[ \frac{\card\big(B(1)\cap[0,\dots,n]\big)}{n+1}\geq \delta-\sum_{p\geq 2}\frac{\delta}{2^p}\geq\frac\delta2 \]
so that $B(1)$ still has positive lower density.
\end{proof}

\begin{proof}[Proof of Theorem \ref{thm:setsafhc}]
Applying Lemmas \ref{lem:setsafhc2} and \ref{lem:setsafhc3},
we start from a sequence $(A(p))$ of pairwise disjoint subsets of $\NN$, satisfying properties (i) and (iii) of Theorem \ref{thm:setsafhc} and property (ii) when $p=q$.
We construct by induction on $r$ sets $B(1),\dots,B(r)$, $A_r(k)$ for $k\geq r+1$ such that
\begin{itemize}
\item $B(k)\subset A(k)$ for all $k\leq r$, $A_r(k)\subset A(k)$ for all $k\geq r+1$;
\item $B(k)$ and $A_r(k)$ have positive lower density;
\item for all $p\in\{1,\dots,r\}$, for all $q\geq p+1$, for all $n\in A(p)$, for all $n'\in B(q)$ if $q\leq r$, for all $n'\in A_r(q)$ if $q\geq r+1$, $|n-n'|\geq a(p)+a(q)$. 
\end{itemize}
The sequence $(B(p))$ that we get at the end will answer the problem.
Now the construction is easily done by successive applications of Lemma \ref{lem:setsafhc4} first with the sequence $(A(p))_{p\geq 1}$, then with the sequence $(A_1(p))_{p\geq 2}$, and so on.
\end{proof}

\subsection{A weighted shift with a frequently hypercyclic algebra on $c_0$}

Let us now define a weight $(w_n)$ such that $B_w$, acting on $c_0$ endowed with the coordinatewise product, supports a frequently hypercyclic algebra. 
We start with the sequence $(A(p))_{p\geq 1}$ given by Theorem \ref{thm:setsafhc} for $a_p=p$. 
We then construct inductively a sequence of integers $(M_k)$ such that, for all $(n,n')\in A(p)\times A(q)$ with $p\neq q$, $\max(n,n')\geq M_{k+1}\implies |n-n'|\geq M_k$. 
This follows directly from property (iii) of Theorem \ref{thm:setsafhc}, applied successively with $C=M_1=1$ to get $M_2$, $C=M_2$ to get $M_3$, and so on.
We may also assume that the sequence $(M_{k+1}-M_k)$ is non-decreasing.

We define the weight $(w_n)_{n\geq 1}$ by the following inductive formulas:
\begin{itemize}
\item $w_n=2$ for all $n\leq M_2$;
\item for all $k\geq 2$, for all $n\in\{M_k+1,\dots,M_{k+1}\}$, 
$$w_n=\left(w_1\cdots w_{M_k}\right)^\frac1{k(M_{k+1}-M_k)}$$
so that, and this is the crucial point, 
$$w_{M_k+1}\cdots w_{M_{k+1}}=\left(w_1\cdots w_{M_k}\right)^{\frac 1k}.$$
\end{itemize}

Let us summarize the properties of the weight  which will be useful later.
\begin{lemma}\label{lem:afhcws1}
The weight $(w_n)$ satisfies the following properties:
\begin{itemize}
\item for all $n\geq 1$, $w_n\geq 1$;
\item $(w_n)$ is non-increasing;
\item $(w_1\cdots w_n)$ tends to $+\infty$;
\item for all $\alpha>0$, for all $l\geq 0$, 
$\displaystyle \frac{w_{M_{k-1}+l+1}\cdots w_{M_{k+1}+l}}{\left(w_{l+1}\cdots w_{M_{k+1}+l}\right)^\alpha}\xrightarrow{k\to+\infty}0. $	
\end{itemize}
\end{lemma}
\begin{proof}
The first property is clear. For the second one, it suffices to prove that if $n\in\{M_k+1,\dots,M_{k+1}\}$ and $n'\in \{M_{k+1}+1,\dots,M_{k+2}\}$ for some $k$, then $w_{n'}\leq w_n$. We now write
\begin{eqnarray*}
w_{n'}&=&\left(w_1\cdots w_{M_{k+1}}\right)^{\frac{1}{(k+1)(M_{k+2}-M_{k+1})}}\\
&=&\left(w_1\cdots w_{M_{k}}\right)^{\frac{1}{(k+1)(M_{k+2}-M_{k+1})}}\left(w_{M_k+1}\cdots w_{M_{k+1}}\right)^{\frac{1}{(k+1)(M_{k+2}-M_{k+1})}}\\
&=&\left(w_1\cdots w_{M_k}\right)^{\frac{1}{k(M_{k+2}-M_{k+1})}}\\
&\leq&w_n.
\end{eqnarray*}
To prove that $(w_1\cdots w_n)$ tends to $+\infty$, we just observe that, for all $k\geq 2$, $$w_1\cdots w_{M_k}=(w_1\cdots w_{M_2})^{\prod_{j=2}^{k-1}\left(1+\frac 1j\right)},$$
and this goes to $+\infty$ as $k$ tends to $+\infty$. Finally, since $(w_n)$ is bounded and bounded below, we need only to prove the last property for $l=0$. Now we write
\begin{align*}
w_{M_{k-1}+1}\cdots w_{M_{k+1}}&=w_{M_{k-1}+1}\cdots w_{M_k}\left(w_1\cdots w_{M_k}\right)^{1/k}\\
&=\left(w_{M_{k-1}+1}\cdots w_{M_k}\right)^{1+\frac 1k}
\left(w_1\cdots w_{M_{k-1}}\right)^{\frac 1k} \\
&=\left(w_1\cdots w_{M_{k-1}}\right)^{\frac 1{k-1}\left(1+\frac 1k\right)+\frac 1k}\\
&=\left(w_1\cdots w_{M_{k-1}}\right)^{\frac 2{k-1}} 
\end{align*}
so that 

$$\frac{w_{M_{k-1}+1}\cdots w_{M_{k+1}}}{\left(w_1\cdots w_{M_{k+1}}\right)^\alpha}
\leq \frac1{\left(w_1\cdots w_{M_{k-1}}\right)^{\alpha-\frac{2}{k-1}}}$$
which indeed tends to zero.
\end{proof}

We now prove that the operator $B_w$ acting on $c_0$ endowed with the coordinatewise product supports a frequently hypercyclic algebra.
Let $(v(p),m(p))$ be a sequence dense in $c_0\times\NN$
such that each $v(p)$ has finite support contained in $[0,p]$. We shall need a last technical lemma involving all the objects we constructed until now.

\begin{lemma}\label{lem:afhcws2}
There exists a sequence of integers $(N(r))_{r\geq 1}$ satisfying the following properties:
\begin{enumerate}[(i)]
\item for all $r\geq 1$, 
$$\sup_{n\geq N(r),\ l=0,\dots,r}\left|\frac{v_l(r)}{\left(w_{l+1}\cdots w_{n+l}\right)^{\frac{1}{m(r)+1}}}\right|^{\frac 1{m(r)}}<\frac 1r.$$
\item for all $r\geq 2$, for all $s\in\{1,\dots,r-1\}$, for all $(j,j')\in A(r)\times A(s)$ with $j\geq N(r)$, for all $l\in \{0,\dots,r\}$, for all $\alpha\geq\min\left(\frac 1{m(r)},\frac 1{m(s)}\right)$, 
\begin{align*}
j>j'&\implies \left|\frac{w_{l+(j-j')+1}\cdots w_{j+l}v_l(r)^\alpha}{\left(w_{l+1}\cdots w_{j+l}\right)^\alpha}\right|<\frac1r\\
j'>j&\implies \left|\frac{w_{l+(j'-j)+1}\cdots w_{j'+l}v_l(s)^\alpha}{\left(w_{l+1}\cdots w_{j'+l}\right)^\alpha}\right|<\frac1r.
\end{align*}
\end{enumerate}
\end{lemma}
\begin{proof}
Let $r\geq 1$ be fixed, We first observe that it is easy to ensure (i), just by assuming that $N(r)$ is large enough. Let us choose $N(r)$ to ensure (ii).
Upon taking a supremum, we may fix $s$ and $l$ and to simplify the notations, we will assume $l=0$. Let $\alpha_0=
\min\left(\frac 1{m(r)},\frac 1{m(s)}\right)$ and $C=\max(1,|v_0(r)|,|v_0(s)|)$.
We define three integers $N_0$, $k_0$ and $k_1$ satisfying the following three conditions:
$$n\geq N_0\implies \frac{C^2}{w_1\cdots w_n}<\frac 1r$$
$$n\geq M_{k_0}\implies \frac{C^2}{w_1\cdots w_n}<\frac 1r$$
$$k\geq k_1\implies \frac{w_{M_{k-1}+1}\cdots w_{M_{k+1}}C}{\left(w_1\cdots w_{M_{k+1}}\right)^{\alpha_0}}
< \frac 1{r}.$$
We set $N(r)=\max(N_0,M_{k_0+1},M_{k_1})$. 
Let $(j,j')\in A(r)\times A(s)$ with $j\geq N(r)$. To fix the ideas, we assume that $j>j'$. If $\alpha\geq 2$, then 
$$\left|\frac{w_{(j-j')+1}\cdots w_{j}v_0(r)^\alpha}{\left(w_{1}\cdots w_{j}\right)^\alpha}\right|
\leq \frac{w_{(j-j')+1}\cdots w_{j}}{w_1\cdots w_j}\times \left|\frac{v_0(r)^2}{w_1\cdots w_j}\right|^{\alpha/2}<1\times\left(\frac 1r\right)^{\frac\alpha 2}\leq \frac 1r.$$
 If $\alpha\in [1,2]$, then
$$\left|\frac{w_{(j-j')+1}\cdots w_{j}v_0(r)^\alpha}{\left(w_{1}\cdots w_{j}\right)^\alpha}\right|
\leq \frac{C^2}{w_1\cdots w_{j-j'}}.
$$
Since $j\geq N(r)$, $j-j'\geq M_{k_0}$ so that the last term is less than $1/r$. 
For $\alpha<1$, since $j\geq N(r)\geq M_{k_1}$,  there exists a single integer $k\geq k_1$ such that $j\in [M_k,M_{k+1})$. Then $j-j'\geq M_{k-1}$ and
\begin{align*}
\left|\frac{w_{(j-j')+1}\cdots w_j v_0(r)^\alpha}{\left(w_1\cdots w_{j}\right)^{\alpha}}\right|
&\leq \frac{w_{M_{k-1}+1}\cdots w_{j} C}{\left(w_1\cdots w_{j}\right)^{\alpha_0}}\\
&\leq \frac{w_{M_{k-1}+1}\cdots w_{M_{k+1}}C}{\left(w_1\cdots w_{M_{k+1}}\right)^{\alpha_0}}\\
&<\frac 1r.
\end{align*}
\end{proof}

\begin{proof}[Proof of Theorem \ref{thm:afhcc0}]
We are now ready for the proof that $B_w$ supports a frequently hypercyclic algebra. 
By Lemma \ref{lem:setsafhc3}, for each $p\geq 1$, let $B(p)$ be a subset of $A(p)$ with positive lower density such that $\min(B(p))\geq N(p)$ and $|n-n'|\geq N(p)$ for all $n\neq n'\in B(p)$. We set
$$u(p)=\sum_{n\in B(p)}\sum_{l=0}^p \frac1{\left(w_{l+1}\cdots w_{n+l}\right)^{1/m(p)}} v_l(p)^{1/m(p)} e_{n+l}.$$

Since $(w_1\cdots w_n)$ tends to $+\infty$, $u(p)$ belongs to $c_0$. Moreover, the choice of $N(p)$ 
(here, (i) of Lemma \ref{lem:afhcws2}) ensures that $\|u(p)\|<1/p$. 
We also observe that the $u(p)$ have pairwise disjoint support.
Hence we may define $u=\sum_{p\geq 1}u(p)$ which still belongs to $c_0$.  We claim that the following property is true:
for all $p\geq 1$, for all $q\neq p$, for all $n\in B(p)$, 
\begin{equation}\label{eq:afhcws1}
 \left\|B_w^n u(p)^{m(p)}-v(p)\right\|<\frac 1p,
\end{equation}
\begin{equation}\label{eq:afhcws2}
 \forall m>m(p),\ \left\|B_w^n u(p)^{m}\right\|<\frac 1p,
\end{equation}
\begin{equation}\label{eq:afhcws3}
 \forall m\geq m(p),\ \left\|B_w^n u(q)^{m}\right\|<\frac 1p.
\end{equation}
Assume that these properties have been proved. Let $P$ be a non-constant polynomial with $P(0)=0$ and write it $P(z)=\sum_{m=m_0}^{m_1} \hat P(m) z^m$ with $\hat P(m_0)\neq 0$.
We aim to prove that $P(u)$ is a frequently hypercyclic
vector for $B_w$. Without loss of generality, we can assume $\hat P(m_0)=1$. 
Let $V$ be a non-empty open subset of $c_0$.
There exists $p\geq 1$ such that $B\left(v(p),\left(2+\sum_{m=m_0+1}^{m_1}|\hat P(m)|\right)/p\right)\subset V$
and $m(p)=m_0$. Then, for all $n\in B(p)$, 
\begin{align*}
 \left\|B_w^n P(u)-v(p)\right\|&\leq \left\|B_w^n u(p)^{m(p)}-v(p)\right\|+\left\|\sum_{q\neq p} B_w^n u(q)^{m(p)}\right\|\\
 &\quad\quad\quad+\sum_{m=m_0+1}^{m_1}|\hat P(m)|\left\|\sum_{q\geq 1} B_w^n u(q)^m\right\|\\
&\leq \frac{2+\sum_{m=m_0+1}^{m_1}|\hat P(m)|}{p},
\end{align*}
where the last inequality follows from \eqref{eq:afhcws1}, \eqref{eq:afhcws2}, \eqref{eq:afhcws3} and the fact that the $B_w^n u(q)$ have pairwise disjoint support.
Therefore, for all $n$ in a set of positive lower density, $B_w^n P(u)$ belongs to $V$, showing that $P(u)$ is a frequently hypercyclic vector for $B_w$.
Hence, it remains to prove \eqref{eq:afhcws1}, \eqref{eq:afhcws2} and \eqref{eq:afhcws3}. We first observe that 
\[ B_w^n u(p)^{m(p)}-v(p)=\sum_{\substack{n'\in B(p)\\ n'\geq n}}\sum_{l=0}^p \frac{v_l(p)}{w_{l+1}\cdots w_{(n'-n)+l}} e_{(n'-n)+l}. \]
Since $n'-n>N(p)$ for all $n'>n$, $n'\in B(p)$, \eqref{eq:afhcws1} follows from (i) in Lemma \ref{lem:afhcws2}. 
Next, for $m>m(p)$, we may write $B_w^n u(p)^m$ as
$$\sum_{\substack{n'\in B(p)\\ n'\geq n}}\sum_{l=0}^p \frac{v_l(p)^{\frac{m}{m(p)}}}{\left(w_{l+1}\cdots w_{(n'-n)+l}\right)^{\frac m{m(p)}}\left(w_{(n'-n)+l+1}\cdots w_{n'+l}\right)
^{\frac{m}{m(p)}-1}}e_{(n'-n)+l}.
$$
There is an additional difficulty since now we may have $n'=n$. We overcome this difficulty by writing
\begin{align*}
&\left|\frac{v_l(p)^{\frac{m}{m(p)}}}{\left(w_{l+1}\cdots w_{(n'-n)+l}\right)^{\frac m{m(p)}}\left(w_{(n'-n)+l+1}\cdots w_{n'+l}\right)^{\frac{m}{m(p)}-1}}\right|\\
&\quad\quad\leq
 \left| \frac{v_l(p)^{\frac{m}{m(p)}}}{\left(w_{l+1}\cdots w_{n'+l}\right)^{\frac{m}{m(p)}-1}}\right|\\
 &\quad\quad\leq \left| \frac{v_l(p)^{\frac{1}{m(p)}}}{\left(w_{l+1}\cdots w_{n'+l}\right)^{\frac{1}{m(p)}-\frac 1m}}\right|^m\\
 &\quad\quad\leq  \left| \frac{v_l(p)^{\frac{1}{m(p)}}}{\left(w_{l+1}\cdots w_{n'+l}\right)^{\frac{1}{m(p)(m(p)+1)}}}\right|^m\\
 &\quad\quad<\frac 1p.
\end{align*}

Finally, for $m\geq m(p)$ and $q\neq p$, we write
\[ B_w^n u(q)^m=\sum_{\substack{n'\in B(q)\\ n'>n}}\sum_{l=0}^q \frac{w_{l+(n'-n)+1}\cdots w_{n'+l}}{\left(w_{l+1}\cdots w_{n'+l}\right)^{\frac m{m(q)}}}v_l(q)^{\frac m{m(q)}}e_{(n'-n)+l}. \]

For $q>p$, we apply (ii) of Lemma \ref{lem:afhcws2} with $r=q$, $s=p$, $j=n'$, $j'=n$ and $\alpha=m/m(q)$. For $q<p$, we apply (ii) of Lemma \ref{lem:afhcws2} with
$r=p$, $s=q$, $j=n$, $j'=n'$ and $\alpha=m/m(q)$. In both cases, we immediately find that all the coefficients of $B_w^n u(q)^m$ are smaller than $1/p$, yielding
\[ \left\|B_w^n u(q)^m\right\|<\frac 1p. \]
This closes the proof of Theorem \ref{thm:afhcc0}.
\end{proof}

This technical construction leads to an example over the not so difficult space $c_0$, but the following question remains open.

\begin{question}
Does there exist a weighted shift on $\ell_p$ endowed with the pointwise product admitting a frequently hypercyclic algebra?
\end{question}

Of course, it would also be nice to get simpler examples! On the other hand, for sequence spaces endowed with the convolution product,
we have neither positive nor negative examples. For instance, it would be very interesting to solve the following questions.

\begin{question}
 Does $B$ on $\omega$ endowed with the convolution product support a frequently hypercyclic algebra?
\end{question}

\begin{question}
Does $2B$ on $\ell_1$ endowed with the convolution product support a frequently hypercyclic algebra?
\end{question}

\section{Concluding remarks and open questions}

\subsection{Closed hypercyclic algebras}

As pointed out in the introduction, provided $T$ is hypercyclic, $HC(T)\cup\{0\}$ always contains a dense subspace. When moreover $T$ satisfies
the hypercyclicity criterion, there is a necessary and sufficient condition to determine whether $HC(T)\cup\{0\}$ contains an infinite-dimensional closed subspace
(see for instance \cite{BM09}). In our context, it is natural to ask whether, for some of our examples, $HC(T)\cup\{0\}$ contains a closed non-trivial algebra
(we will say that $T$ supports a closed hypercyclic algebra).

The third author and K. Grosse-Erdmann have shown that it is the case if $T$ is a translation operator acting on the space $\mathcal C^{\infty}(\mathbb R,\mathbb C)$.
The fact that $T$ is an algebra homomorphism plays an important role here. We now give several negative results. The first one solves a question of \cite{shkarin}.

\begin{proposition}
No convolution operator $P(D)$ induced by a nonconstant polynomial $P\in\CC[z]$ admits a closed hypercyclic algebra.
\end{proposition}
\begin{proof}
We write  $P(z)=\sum_{s=0}^t\hat{P}(s)z^s$, with $\hat{P}(t)\neq0$, and let $f\in HC(P(D))$. We shall prove that the closed algebra generated by $f$ contains a non-zero and non-hypercyclic vector. Write $f(z)=a_0+\sum_{n\geq p}a_n z^n$, with $a_p\neq 0$. Without loss of generality, we may assume that $a_p=1$. We shall construct by induction a sequence of complex numbers $(b_k)$ such that 
\[ |b_k|\leq  \left(\frac{|\hat P(0)|+1}{|\hat{P}(t)|}\right)^{kp}\times\frac 1{(ktp)!}\]
 for all $k$ and, setting $P_k(z)=\sum_{l=1}^k b_l ( z-a_0)^{lt}$, then 
 \[ |P(D)^{lp}(P_k\circ f)(0)|\geq (|\hat P(0)|+1)^{lp} \]
 for all $1\leq l\leq k$. The conclusion follows easily. In fact, $(P_k)$ converges uniformly on compact subsets of $\CC$ to some entire function $g$. From the uniformity of the convergence, we conclude that the function $g\circ f$ satisfies 
 \[ \left|P(D)^{lp}(g\circ f)(0)\right|\geq (|\hat P(0)|+1)^{lp}\] 
 for all $l\geq 1$.  Let us set $h=g-g(0)$. The function $h\circ f$, which is in the algebra generated by $f$, satisfies
\begin{align*}
\left |P(D)^{lp}(h\circ f)(0)\right| & \geq \left|P(D)^{lp}(g\circ f)(0)\right| - \left| P(D)^{lp}(g(0))\right|\\
&\geq  (|\hat P(0)|+1)^{lp} - |\hat P(0)|^{lp} |g(0)|\\
&\xrightarrow{l\to+\infty}+\infty.
\end{align*} 
Hence, $h\circ f$ is nonzero and it cannot be hypercyclic for $P(D)^p$. In particular, since $HC(P(D))\subset \bigcap_{n\geq1}HC(P(D)^n)$ (see \cite[Theorem 1]{ansari}), $h\circ f$ cannot be hypercyclic for $P(D)$ as well.

For the proof we will use the formula
\[
P(D)^q=\sum_{{\bf j}\in I_q}{\binom q {\bf j}}\hat{P}(0)^{j_0}\cdots\hat{P}(t)^{j_t}D^{j_1+2j_2+\cdots+tj_t},
\]
where $I_q=\{{\bf j}=(j_0,...,j_t)\in \NN_0^{t+1} : j_0+\cdots+j_t=q\}$ and $\binom q {\bf j}$ denote the multinomial coefficient \[\binom q {j_0,\dots,j_t}=\frac{q!}{j_0!\cdots j_t!}.\]
 
Let us set $P_0(z)=0$ and let us assume that the construction has been done until step $k-1$. Then
\begin{align*}
P_{k-1}\circ f+b(f-a_0)^{kt}&=P_{k-1}\circ f+b(z^p+a_{p+1}z^{p+1}+\cdots)^{kt}\\
  &=P_{k-1}\circ f+b z^{ktp}+\sum_{j\geq ktp+1}c_j z^j.
\end{align*}

Hence, for $1\leq l\leq k$,
\begin{align*}
P(D)^{lp}(P_{k-1}\circ f+b(f-a_0)^{kt})(0)&=P(D)^{lp}(P_{k-1}\circ f)(0)+P(D)^{lp}(b(f-a_0)^{kt})(0)\\ &=P(D)^{lp}(P_{k-1}\circ f)(0)+g_l(0),
\end{align*}
where
\begin{align*}
g_l(z)
&=P(D)^{lp}\left(bz^{ktp}+\sum_{j\geq ktp+1}c_jz^j\right).
\end{align*}
If $l\leq k-1$ then $\deg P^{lp}\leq(k-1)tp<ktp$, hence $g_l(0)=0$. By the induction hypothesis it follows that
\begin{align*}
|P(D)^{lp}(P_{k-1}\circ f+b(f-a_0)^{kt})(0)|&=|P(D)^{lp}(P_{k-1}\circ f)(0)+g_l(0)|\\
&=|P(D)^{lp}(P_{k-1}\circ f)(0)|\geq (|\hat P(0)|+1)^{lp}
\end{align*}
whatever the value of $b$ is. On the other hand, if $l=k$, then
\begin{align*}
g_k(z)
&=b\hat{P}(t)^{kp}(ktp)!\\ &\quad\quad\quad+\sum_{{\bf j}\in I_{kp}\backslash\{(0,...,0,kp)\}}\binom{kp}{\bf j}\hat{P}(0)^{j_0}\cdots\hat{P}(t)^{j_t}D^{j_1+2j_2+\cdots+tj_t}\left(bz^{ktp}+\sum_{j\geq ktp+1}c_jz^j\right),
\end{align*}
hence $g_k(0)=b\hat{P}(t)^{kp}(ktp)!$, that is,
\begin{align*}
|P(D)^{kp}(P_{k-1}\circ f+b(f-a_0)^{kt})(0)|&=|P(D)^{kp}(P_{k-1}\circ f)(0)+g_k(0)|\\
&= |P(D)^{kp}(P_{k-1}\circ f)(0)+b\hat{P}(t)^{kp}(ktp)!|,
\end{align*}
so we can find $b$ satisfying 
\[ |b|\leq  \left(\frac{|\hat P(0)|+1}{|\hat{P}(t)|}\right)^{kp}\times\frac 1{(ktp)!}\]
such that 
 \[ |P(D)^{lp}(P_{k-1}\circ f+b(f-a_0)^{kt})(0)|\geq (|\hat P(0)|+1)^{lp}. \]
  The proof is now done by taking $b_k=b$.
\end{proof}

\begin{question}
Does there exist an entire function $\phi$ of exponential type such that $\phi(D)$ supports a closed hypercyclic algebra?
\end{question}

\begin{proposition}
Let $X=\ell_p$, $X=c_0$ or $X=\omega$, endowed with the coordinatewise product. No backward shift on $X$ supports a closed hypercyclic algebra.
\end{proposition}

\begin{proof}
We first consider the case $X=c_0$. Let $x\in X$ be a non-zero sequence, and let $D\subset \mathbb{C}$ be a compact disc centered at the origin and
omitting at least one of the terms of $x$.  For each $n\in \NN$, consider the compact set $K_n$ as in Figure \ref{polyfigure} and $f$ a holomorphic function
defined on a neighbourhood of $K_n$ and satisfying that $f(z)=0$ if $z\in D$ and $f(z)=1$ if $z\in K_n\setminus D$. 

\begin{figure}[ht]
    \centering
    \includegraphics[height=7cm]{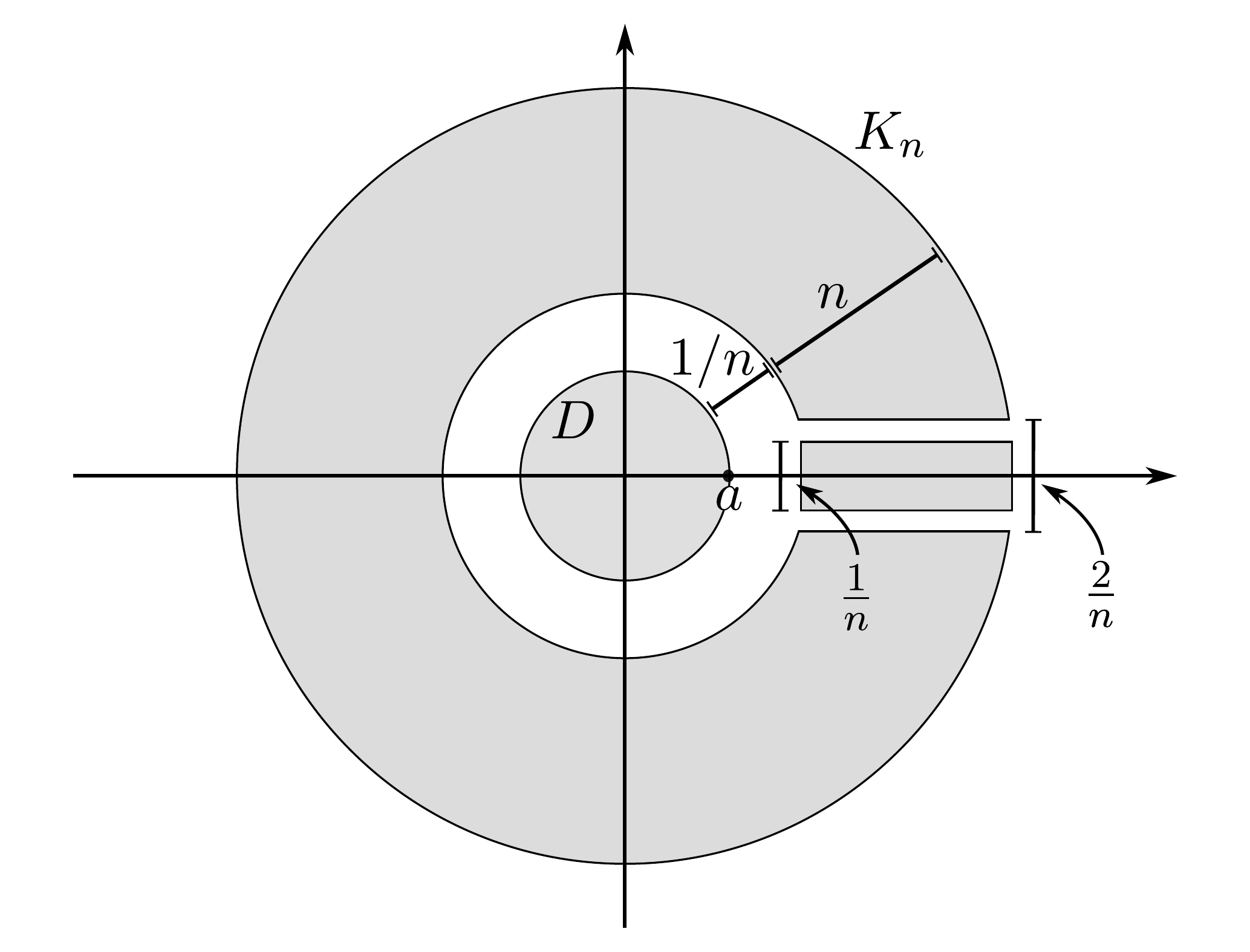}
    \caption{}
    \label{polyfigure}
\end{figure}

By Runge's approximation theorem we get a polynomial $P_n$ such that $\| P_n-f\| _{K_n}<\frac{1}{n}$. We end up with a sequence of polynomials $(P_n)$,  satisfying that $P_n(z)\rightarrow 0$ uniformly on $D$, and $P_n(z)\rightarrow 1$ pointwise on $\mathbb{C}\setminus D$. Redefining $P_n$ by $P_n-P_n(0)$, we may also assume that $P_n(0)=0$, for every $n\in \NN$. 

Since $x\in c_0$, it follows that eventually all the terms of $x$ belong to $D$ which yields that $P_n(x)\rightarrow y=(y_k)$ in $c_0$, where $y_k=0$, if $x_k\in D$ and $y_k=1$ otherwise. We conclude,  that $y$ is a non-zero element in the closed algebra generated by $x$ which is not hypercyclic for any weighted backward shift on $c_0$.  

Let now $X=\ell ^p, p\geq 1$. Consider an $x\in X, x\neq 0$, and $D$ and $(P_n)$ defined as above. Cauchy's formula ensures that $P_n'\rightarrow 0$ uniformly on $\frac{1}{2}D$. Let
$$
C=\sup \left\{ |P_n'(z)|: n\in \NN, z\in \frac{1}{2}D \right\}.
$$
Fix $\varepsilon >0$ and let $k_0\in \NN$ be such that, for $k>k_0$, $x_k\in\frac 12 D$ and  
$$
2^pC^p\sum_{k>k_0}|x_k|^p<\frac{\varepsilon}{2}.
$$
Find $N\in \NN$ such that for all $m,n\geq N$,
$$
\sum_{k=1}^{k_0}|P_m(x_k)-P_n(x_k)|^p<\frac{\varepsilon}{2}.
$$
We have
\begin{align*}
\|P_m(x)-P_n(x)\|_p^p &\leq \sum_{k=1}^{k_0}|P_m(x_k)-P_n(x_k)|^p+\sum_{k>k_0}(|P_m(x_k)|+|P_n(x_k)|)^p  \\
&\leq \sum_{k=1}^{k_0}|P_m(x_k)-P_n(x_k)|^p+ 2^pC^p\sum_{k>k_0}|x_k|^p<\varepsilon.
\end{align*}
That means that the sequence $(P_n(x))$ is Cauchy in $\ell_p$ and the conclusion follows exactly as in the previous case.

Finally, we consider the case $X=\omega$. Letting $D=\{0\}$ and $K_n$ be as above, and by using Runge's approximation theorem, we get a sequence of polynomials $(Q_n)$ such  that $Q_n(0)=0$, for every $n\in \NN$ and $Q_n(z)\rightarrow 1$ for each $z\in \mathbb{C}\setminus \{0\}$. If $x\in \omega$ is a non-zero sequence, then $Q_n(x)\rightarrow y=(y_k)$, where $y_k=0$ if $x_k=0$, and $y_k=1$ otherwise. It is immediate that $y$ is a non-zero element in the closed algebra generated by $x$ which fails to be hypercyclic for any weighted backward shift on $\omega$.
\end{proof}

\begin{question}
 Does there exist a weight $(w_n)$ such that $B_w$, acting on $\ell_1$ endowed with the Cauchy product, supports a closed hypercyclic algebra? 
\end{question}

\subsection{Hypercyclic algebras in the ideal of compact operators}

Beyond the examples given in that paper, there are other examples where the existence of a hypercyclic algebra would be natural. One of them is given by hypercyclic operators acting on separable ideals of operators. For instance, assume that $H$ is a separable Hilbert space and denote by $X=\mathcal K(H)$ the (non-commutative) algebra of compact operators in $H$, endowed with the norm topology.

For $T\in\mathcal L(H)$, denote by $L_T$ the operator of left multiplication by $T$, defined on $\mathcal K(H)$.
It is known (see for instance \cite[Chapter 8]{BM09}) that if $T$ satisfies the hypercyclicity criterion, then $L_T$ is a hypercyclic operator on $\mathcal K(H)$. This latter space being an algebra, it is natural to study whether $L_T$ supports a hypercyclic algebra. We do not know the answer to this question, but we point out that a positive answer would require different techniques. Indeed, Theorem \ref{thm:generalcriterion} can never be applied to these operators.

\begin{proposition}
Let $T\in\mathcal L(H)$. Then $L_T$, acting on $\mathcal K(H)$, does not satisfy the assumptions of Theorem \ref{thm:generalcriterion} even for $d=1$.
\end{proposition}

\begin{proof}
We fix $x\in H$, $x^*\in H^*$ with $x^*(x)=1$, $\|x\|=1$, $\|x^*\|=1$. 
Using the notations of Theorem \ref{thm:generalcriterion}, let $A=\{1,2\}$, 
\begin{align*}
U=V&=\left\{u\in \mathcal L(H):\ \|u-x^*\otimes x\|<1/4\right\},\\
W&=\left\{u\in \mathcal L(H):\ \|u\|<1/8\right\}.
\end{align*}
Assume first that $\beta=1$ and that there exist $u\in U$, $N\in\mathbb N$ with $T^N u\in V$ and $T^N u^2\in W$. Then we know that
\[ \left\|T^N u^2(x)-x^*(u(x))x\right\|<\frac{\|u(x)\|}4 \]
since $T^N u\in V$. Now, $\|u(x)-x\|<1/4$ so that $\|u(x)\|<5/4$ and $|x^*(u(x))|> 3/4$. Hence, 
\[ \left\|T^N u^2(x)\right\| > \frac 34-\frac 5{16}>\frac 18. \]
This contradicts $T^N u^2\in W$.

If we assume that $\beta=2$ and that there exist $u\in U$, $N\in\mathbb N$ with $T^N u\in W$ and $T^N u^2\in V$, then we get successively
\begin{align*}
\left\| T^N u^2(x)-x\right\|&<\frac 14\textrm{ (since $T^N u^2\in V$)}\\
\left\| T^N u^2(x)\right\|&<\frac 18 \|u(x)\|\textrm{ (since $T^N u\in W$)}\\
\left\|u(x)-x\right\|&<\frac 14.
\end{align*}
These three inequalities yield easily a contradiction.
\end{proof}

\subsection{Further question and remark}

As independently shown by Ansari in \cite{ansariexistence} and later-on by Bernal-Gonzáles in \cite{gonzalez}, every separable Banach space supports a hypercyclic operator. In the context of algebras a natural question arises.
 \begin{question} Is it true that every separable Banach algebra supports a hypercyclic operator admitting a hypercyclic algebra? \end{question}

In all known results, the set of generators for hypercyclic algebras is either empty or residual. We observe below that this set can be non-empty and meager.

\begin{remark}
For every pair $(X,T)$ where $X$ is a Banach space and $T$ a hypercyclic operator with a non-hypercyclic vector, we may define a product on $X$ turning it into a commutative Banach algebra and such that the set of generators for a hypercyclic algebra for $T$ is non-empty and nowhere dense.
\end{remark} 

\begin{proof}
Let $x\in HC(T)$ and $y$ a non-hypercyclic vector for $T$ with $\|y\|=1$. Consider $f\in X^{\ast}$ with $\|f\|=1$ and such that $f(x)=0$ and $f(y)=1$. We define the product 
$$
z\cdot w=f(z)f(w)y, \,\, \mbox{with} \,\,  z,w \in X
$$
and observe that it turns $X$ into a commutative Banach algebra. Now, $x^2=0$ so $A(x)=span(x)$, and thus $x$ is a generator for a hypercyclic algebra for $T$. Moreover, it is easy to check that the following holds,
$$
\{x\in X: A(x)\setminus \{0\}\subset HC(T)\}=HC(T)\cap Ker(f).
$$
Since $Ker(f)$ is a proper, closed hyperplane of $X$, we conclude that the set of generators for a hypercyclic algebra for $T$ is non-empty and nowhere dense. 
\end{proof}


\bibliographystyle{plain} 
%

\end{document}